\documentclass{amsart}

\usepackage{amssymb,amsfonts}
\usepackage[all,arc]{xy}
\usepackage{enumerate}
\usepackage{mathrsfs}
\usepackage{hyperref}
\hypersetup{
pdftitle={},
pdfsubject={},
pdfauthor={Kolosov Petro},
pdfkeywords={}
}
\def \dd {\partial}

\DeclareMathOperator{\Hess}{Hess}
\DeclareMathOperator{\Ric}{Ric}

\usepackage[usenames,dvipsnames]{xcolor}

\usepackage{bm}
\usepackage{pgfplots}
\usepackage{verbatim} 

\newtheorem{thm}{Theorem}[section]

\newtheorem{prop}[thm]{Proposition}
\newtheorem{lem}[thm]{Lemma}

\theoremstyle{definition}
\newtheorem{defn}[thm]{Definition}

\theoremstyle{remark}
\newtheorem{rem}[thm]{Remark}

\makeatletter
\let\c@equation\c@thm
\makeatother
\numberwithin{equation}{section}

\makeatletter
\renewcommand{\tocsection}[3]{%
  \indentlabel{\@ifnotempty{#2}{\bfseries\ignorespaces#1 #2\quad}}\bfseries#3}
\renewcommand{\tocsubsection}[3]{%
  \indentlabel{\@ifnotempty{#2}{\ignorespaces#1 #2\quad}}#3}

\newcommand\@dotsep{4.5}
\def\@tocline#1#2#3#4#5#6#7{\relax
  \ifnum #1>\c@tocdepth 
  \else
    \par \addpenalty\@secpenalty\addvspace{#2}%
    \begingroup \hyphenpenalty\@M
    \@ifempty{#4}{%
      \@tempdima\csname r@tocindent\number#1\endcsname\relax
    }{%
      \@tempdima#4\relax
    }%
    \parindent\z@ \leftskip#3\relax \advance\leftskip\@tempdima\relax
    \rightskip\@pnumwidth plus1em \parfillskip-\@pnumwidth
    #5\leavevmode\hskip-\@tempdima{#6}\nobreak
    \leaders\hbox{$\m@th\mkern \@dotsep mu\hbox{.}\mkern \@dotsep mu$}\hfill
    \nobreak
    \hbox to\@pnumwidth{\@tocpagenum{\ifnum#1=1\bfseries\fi#7}}\par
    \nobreak
    \endgroup
  \fi}
\AtBeginDocument{%
\expandafter\renewcommand\csname r@tocindent0\endcsname{0pt}
}
\def\l@subsection{\@tocline{2}{0pt}{2.5pc}{5pc}{}}
\makeatother

\usepackage{lipsum}

\pgfplotsset{compat=1.18}
\title{Modulus of Concavity and  Fundamental Gap Estimates on Surfaces }
\begin{document}

\begin{abstract}
The fundamental gap of a domain is the difference between the first two eigenvalues of the Laplace operator. In a series of recent and celebrated works, it was shown that for convex domains in $\mathbb R^n$ and $\mathbb S^n$ with Dirichlet boundary condition the fundamental gap is at least $\frac{3 \pi^2}{D^2}$ where $D$ denotes the diameter of the domain. 
The key to these results is to establish a strong 
concavity estimate for the logarithm of the first eigenfunction.
In this article, we prove corresponding log-concavity and fundamental gap estimates for surfaces with non-constant positive curvature via a two-point maximum principle. However, the curvature not being constant greatly increases the difficulty of the computation. 
\end{abstract}
 \author[Iowa State University]{Gabriel Khan} 
 \address[Gabriel Khan]{Department of Mathematics, Iowa State University, Ames, IA, USA.}
  \email{gkhan@iastate.edu}
 \thanks{G. Khan is supported in part by Simons Collaboration Grant 849022}
 
   \author[UCSB]{Malik Tuerkoen}
   \address[Malik Tuerkoen]{Department of Mathematics, University of California,  Santa Barbara, CA, USA.}
  \email{mmtuerkoen@ucsb.edu}
  
 \author[]{Guofang Wei}
\address[Guofang Wei]{Department of Mathematics, University of California,  Santa Barbara, CA, USA.}
\email{wei@math.ucsb.edu}
\thanks{G. Wei is partially supported by NSF DMS 2104704. }
\maketitle


\section{Introduction}

Given a bounded domain $\Omega $ in a Riemannian manifold $ (M,g)$, the Laplace operator has an increasing sequence of Dirichlet eigenvalues $0<\lambda_1 < \lambda_2 \leq \lambda_3 \leq \ldots \rightarrow \infty$ and corresponding eigenfunctions $u_i:\Omega\rightarrow \mathbb R$ satisfying
\begin{align}\label{helmholtzequation}
    -\Delta u_i=\lambda_i u_i , \quad \left . {u_i} \right\vert_{\partial \Omega}=0.
\end{align}

The \emph{fundamental gap} is the difference between the first and second eigenvalues
\begin{align*}
    \Gamma(\Omega) = \lambda_2(\Omega)-\lambda_1(\Omega)>0. 
\end{align*}
Obtaining upper and lower bounds for this quantity is important both in mathematics and physics, and has been studied extensively (see, e.g. \cite{singer1985estimate,yu1986lower,lee1987estimate,wang2000estimation, Ni2013, Lu-Rowlett2013, Gong-Li-Luo2016},  the references in \cite{andrews2011proof, daifundamental} and the references in the rest of introduction).
In a celebrated paper, Andrews and Clutterbuck \cite{andrews2011proof} proved the \emph{fundamental gap conjecture} \cite{van1983condensation,yau1986nonlinear,ashbaugh1989optimal}, which states that
for all bounded convex domains $\Omega \subset \mathbb{R}^n$ and Schr\"odinger operators $-\Delta + V$ where $V$ is a convex potential, the fundamental gap satisfies \begin{equation} \Gamma(\Omega) \geq \frac{3\pi^2}{D^2},  \label{gap-3pi} 
\end{equation}
    where $D$ is the diameter of $\Omega$.
The quantity on the right hand side of this estimate corresponds to the fundamental gap of a one-dimensional interval with vanishing potential and is optimal in Euclidean space (See Appendix \ref{Wecannotimproveon3} and Remark \ref{3issharpinEuclideanspace} for more details). 
 Their work uses a two-point maximum principle to prove a strong modulus of concavity estimate for the first eigenfunction.

Several years later, the third named author joint with Dai, He, Seto, and Wang 
(in various subsets) extended the fundamental gap estimate \eqref{gap-3pi} to the round sphere 
\cite{10.4310/jdg/1559786428,he2020fundamental,Dai-Seto-Wei2021}. The presence of curvature greatly complicates the analysis, and these proofs relied in an essential way on the fact that the sectional curvature is \emph{constant}.

 For spaces of negative curvature, the behavior of the fundamental gap turns out to be very different.
In fact, the third named author, Bourni, Clutterbuck, Nguyen, Stancu, and Wheeler \cite{bourni2022vanishing} showed that for any $D>0$ and $\epsilon>0,$ there are convex domains $\Omega \subset \mathbb{H}^n$  of diameter $D$ which satisfy 
 \[ \Gamma(\Omega) <\frac{\epsilon}{D^2}.\] 
 In other words, there is no universal lower bound on the fundamental gap among convex domains of any given diameter.

This result was then extended by Nguyen and the first named author. They showed that for any Riemannian manifold with \emph{any negative sectional curvature}, it is possible to construct convex domains of small diameter whose fundamental gap is \emph{arbitrarily small} \cite{khan2022negative}. As such, obtaining estimates on the fundamental gap in terms of the diameter alone requires that the space has positive curvature. 

Establishing a lower bound on the fundamental gap for convex domains on manifolds with non-constant positive curvature remains a challenge, and is open even for seemingly simple manifolds such as $\mathbb S^2 \times \mathbb S^2$ or $\mathbb {CP}^n$. In \cite{surfacepaper1}, the authors joint with Nguyen obtained the first result in this direction for surfaces with positive curvature using a (one-point) maximum principle, showing that $\Gamma(\Omega) > \frac{\pi^2}{D^2} + \underline \kappa$, under some derivative assumption on the curvature $\kappa$, where $\underline \kappa=\inf _\Omega \kappa.$ 
In this paper, we use two-point maximum principle to study the fundamental gap of surface with non-constant positive curvature. In order to make the two-point maximum principle work, we need to overcome several difficulties, which we will discuss in Subsection \ref{sub1.1}. 
We obtain stronger gap estimates than in our previous work \cite{surfacepaper1} (joint with Nguyen) and, in particular, recover the estimate of \cite{Dai-Seto-Wei2021} for $\mathbb S^2$. 
\begin{thm}\label{main-thm}
     Suppose $(M^2,g)$ is a Riemannian manifold with curvature $\kappa$ bounded by $0< \underline \kappa \le \kappa \le c
     $, there is a constant $\alpha(\underline \kappa,\overline{\kappa})>0$ such that if $$|\nabla \kappa|_\infty \leq \alpha(\underline \kappa,\overline{\kappa}), \ \ \ - (\inf \Delta \kappa)_{-} \le \alpha(\underline \kappa,\overline{\kappa}),$$  
    then for all convex domains $\Omega \subset M^2$ with diameter $D\le \tfrac{\pi}{2\sqrt{\overline \kappa}}$ we have the gap estimate 
    \begin{equation}\label{gap-estimate}
     \Gamma(\Omega)\geq \frac{3\pi^2}{D^2} - (12+3\pi) (\overline{\kappa} -\underline{\kappa}).
\end{equation}
\end{thm} 

The constant $\alpha(\underline \kappa,\overline{\kappa})$ will be given explicitly and goes to zero when $\underline \kappa$ goes to zero (see \eqref{curvature-assumption-in-moc-thm} for the precise condition). 

As the curvature of $M^2$ is $0 <\kappa \le \overline{\kappa}$, we have the injectivity radius of $M^2$ is $\ge \tfrac{\pi}{2\sqrt{\overline \kappa}}$ (see e.g. \cite[Page 281]{doCarmo}). Hence there is no cut point in the interior of $\Omega$.

\subsection{Log-concavity and the maximum principle} \label{sub1.1}

 Following \cite{andrews2011proof,10.4310/jdg/1559786428}, the first key step to proving gap estimates of this size is establishing a super log-concavity estimate for the first eigenfunction. More precisely, if we take $u$ to be first Dirichlet eigenfunction\footnote{ Throughout the rest of this article, we drop the subscript $1$ from the first eigenfunction $u_1$ so that subscripts can be used to indicate derivatives of functions. We also will denote $w=\log u.$} on a convex domain $\Omega$ and define for $w=\log u,$ we show that for $x,y\in \Omega \times \Omega, $ $\mathcal Z(x,y)\leq 0,$ where we set 
 \begin{align}\label{superlogconvaity1}
  \mathcal Z(x,y)= \langle \nabla w (y),\gamma'_{x,y}(\tfrac{d}{2})\rangle - \langle \nabla w (x),\gamma'_{x,y}(\tfrac{-d}{2})\rangle + \frac{\pi}{L}\tan\left(\frac{\pi}{ L}{d(x,y)}\right)-\textup{tn}_{\underline \kappa}\left(\frac{d(x,y)}{2}\right),
  \end{align} 
where $\gamma_{x,y}$ is the unique minimizing geodesic satisfying $\gamma_{x,y}(-\tfrac{d(x,y)}{2})=x$ and $\gamma_{x,y}(\tfrac{d(x,y)}{2})=y$.
In this expression, $L$ is taken to be slightly larger than $D$ and satisfies $L=D+O(\overline \kappa-\underline \kappa)$ as $(\overline \kappa-\underline \kappa)\rightarrow 0^+$. Furthemore, we denote $ \textup{tn}_K(s)=\sqrt{K} \tan (\sqrt{K}s)$ (for $K\geq 0$).  

In other words, we have that  \[F(x,y)=\frac{\pi}{L}\tan\left(\frac{\pi}{ L}{d(x,y)}\right)-\textup{tn}_{\underline \kappa}\left(\frac{d(x,y)}{2}\right) \] is a \emph{modulus of concavity for $w$}, where we call a function $F(x,y)$ a modulus of concavity for $w$ if $\lim_{d(x,y)\rightarrow 0}F(x,y)=0$ and if for all $x,y\in\Omega$
\begin{align*}
    \langle\nabla w(y),\gamma'_{x,y}(\tfrac d2)\rangle- \langle \nabla w(x),\gamma'_{x,y}(-\tfrac d2)\rangle\leq F(x,y).
\end{align*}

Proving $\mathcal Z\leq0$ is the most difficult part of the proof. In order to establish this fact, one is tempted to apply the maximum principle directly to the function $\mathcal Z(x,y)$, which is what was done for $\mathbb R^n$ and $\mathbb S^n.$ However, when the curvature is not constant, there are additional terms that cannot be controlled and thus introduce a significant difficulty, see Lemma~\ref{two-point-differential-inequality}. In order to overcome this difficulty, we instead apply the maximum principle to the function 
\begin{align}\label{Z-intor}
     Z(x,y)= \langle \nabla w (y),\gamma'_{x,y}(\tfrac{d}{2})\rangle - \langle \nabla w (x),\gamma'_{x,y}(\tfrac{-d}{2})\rangle-2 \phi(\tfrac{d}{2})-\mathcal C(x,y),
\end{align}
where $\phi$ will be specified below. Here, $\mathcal C$ is defined in the following way.
\begin{defn}[$\mathcal{C}(x,y)$]
    Suppose we have two points $x,y \in (M^2,g)$ which are not cut points of each other. We let $J_{x,y}$ be such that 
\begin{align}
    J''(s)+\kappa(\gamma_{x,y}(s)) J(s)=0, \quad J(-\tfrac d2)=J(\tfrac d2)=1.  \label{Jacobi-ODE}
\end{align}
$\mathcal{C}(x,y)$ is defined to be the quantity
\begin{equation}\label{def-of-C}
  \mathcal{C}(x,y) =  \frac{1}{2}\left(  J_{x,y}'(\tfrac d2)-J_{x,y}'(-\tfrac d2)\right).
\end{equation}
\end{defn}
Note that $2 \mathcal{C}(x,y)$ is the index form of the Jacobi field $J$ along the geodesic $\gamma_{x,y}$ with $J(-\tfrac d2) = e_x, J(\tfrac d2) =e_y$, where $e_x$ is a unit vector perpendicular to $\gamma'_{x,y}(-\tfrac d2)$ and $e_y$
 is the parallel transport of $e_x$ along $\gamma_{x,y}$. See Lemma~\ref{C-derivative of distance}. 
 
 The quantity $\mathcal{C}(x,y)$ is also related to the  
 the classic Gaussian curvature by taking the limit as $y$ goes to $x$. More precisely, from \eqref{Jacobi-ODE}, we have
    \begin{align}\label{collapse-of-C}
    \frac{-\kappa(x)}{2}=\lim_{y \rightarrow x} \frac{\mathcal{C}(x,y)}{d(x,y)}.
\end{align}

In our computation, $\mathcal{C}(x,y)$ serves as part of the modulus of concavity for $w$, which plays the essential role of absorbing some of the problematic terms (see Proposition~\ref{CancellationProposition}).

\subsection{Overview of the paper}

Because the proof of Theorem \ref{main-thm} requires a number of steps and several involved calculations, in this subsection we provide an overview of the paper and the main ideas used in the argument. In Section \ref{Jacobifieldcomparisonsection}, we establish several fundamental properties for Jacobi fields which will be used throughout the paper. In particular, we give a formula for the solution of inhomogenous Jacobi equations given boundary data. This is used several times in Section \ref{derivativesection} and \ref{DerivativesofCSection}. 
We also prove several comparison results and identities for Jacobi fields which are used in the later sections.

In Section \ref{MOCsection} we apply a two-point maximum principle to $Z$ which is given by  
\begin{align*}
    Z(x,y)=\langle \nabla w(y), \gamma_{x,y}'(\tfrac{d}{2})\rangle -\langle \nabla w(x),\gamma_{x,y}'(-\tfrac d2)\rangle -F(x,y),
\end{align*}
 assuming that it achieves a positive maximum in the interior of $\Omega \times \Omega \backslash \{ (x,x) ~|~ x \in \Omega \}$. From the maximum principle we derive a differential inequality for $F(x,y)$ to satisfy, to make it a modulus of concavity for $w$. We first make no assumptions on the dimension or geometry of the manifold (see Lemma \ref{two-point-differential-inequality}).

We then specialize to the two-dimensional case and focus on functions of the form $F(x,y)=2\phi(\tfrac{d(x,y)}{2})+\mathcal C(x,y).$ The choice of the function $\mathcal C$ here is crucial to make the maximum principle work (see Proposition \ref{CancellationProposition}). We use our maximum principle computation to show that $F(x,y)$ is a modulus of concavity for $w,$ i.e. $Z\leq 0,$ if $\phi$ satisfies a collection of differential inequalities (see Theorem \ref{Propositiongeneralmoc} and Theorem \ref{MOC-thm}). In this derivation, we take for granted several crucial calculations, which are postponed to Sections ~\ref{derivativesection} and \ref{DerivativesofCSection}.

Having derived these inequalities, in Section \ref{perturbed-euclidean-model-section} we introduce the \emph{perturbed Euclidean model}\footnote{This is a perturbation of the model $\phi''+\lambda\phi=0,$ which \cite{Dai-Seto-Wei2021} called \emph{Euclidean model}.}  
\begin{align}\label{perturbed-e-model}
    \overline \phi''-4\left(\textup{tn}_{\overline \kappa}(s)-\textup{tn}_{\underline \kappa}(s)\right)\overline \phi'=-\overline \lambda\overline \phi\quad \textup{with } \overline \phi(-D/2)=\overline \phi(D/2)=0.
\end{align}

We then consider the function \[\phi=(\log \overline \phi_1)'+\textup{tn}_{\underline \kappa}, \] 
 where $\overline \phi_1$ the first eigenfunction of \eqref{perturbed-e-model} and show that this satisfies the differential inequalities required for a modulus of concavity. Appealing to Theorem \ref{MOC-thm-2}, we find that  $F(x,y)=2(\log \overline \phi_1)'(\tfrac{d(x,y)}{2})+\textup{tn}_{\underline \kappa}(\tfrac{d(x,y)}{2})$ is a modulus of concavity for $w$.

In Section \ref{gapcomparison-section}, we compare the fundamental gap $\Gamma(\Omega)$ to the gap of a one-dimensional model. To do so, we first apply a Riccati comparison theorem to conclude that $\mathcal Z\leq 0$ and then compare the gap to that of a Euclidean model on a slightly larger interval of length $L$, where $L=D+O(\overline \kappa-\underline \kappa)>D$\begin{align*}
     \phi''+\lambda\phi=0 \quad \textup{in }[-L/2,L/2], \quad \phi(-L/2)=\phi(L/2)=0.
 \end{align*}
Doing so, we prove the estimate \begin{align*}
     \Gamma(\Omega) \geq \frac{3\pi^2}{L^2}.
 \end{align*}
Estimating $L$ in terms of $D, \overline \kappa,$ and $ \underline \kappa$ then gives the final gap estimate \eqref{gap-estimate}.

In order to finish the proof, we must go back and derive several of the key results which were used in the maximum principle argument. In Section \ref{derivativesection}, we prove the crucial result (Proposition~\ref{CancellationProposition}) which shows how the derivatives of $\mathcal C$ absorb the terms which cannot be controlled and thus allow for the maximum principle to be applied. The proof relies on a calculation of Jacobi fields, and uses the computations done in Section \ref{Jacobifieldcomparisonsection}.

Applying the maximum principle to $Z$ also involves the second derivatives of $\mathcal{C}(x,y)$, which are computed in Section \ref{DerivativesofCSection}. The calculations are similar to those of Section \ref{derivativesection}, but more involved and involve derivatives of the curvature, Jacobi fields and its derivatives (see Proposition~\ref{E_1+E_2-derivatives-of-C}).  With enough control on the first and second derivatives of the curvature, it is possible to control all these quantities.

Finally, in the appendix we show that the constant $3\pi^2$ in  the gap estimate
\[\Gamma(\Omega)\geq \frac{3\pi^2}{D^2} - (12+3\pi) (\overline{\kappa} -\underline{\kappa}) \]
cannot be increased, independently of $D.$

\begin{rem}
    We use dimension two in two ways. First the Jacobi equation becomes one scalar ODE instead of system of coupled ODEs. Second, even when the Jacobi equation is a decoupled system, like in the case of $\mathbb S^2 \times \mathbb S^2$ or $\mathbb {CP}^n$, one cannot control some of the terms as it is not a scalar matrix (see Section \ref{2-d-subsection} and the paragraph above that).
\end{rem}

  \section{Jacobi fields and their comparison geometry}\label{Jacobifieldcomparisonsection}

In this section we establish some fundamental properties of Jacobi fields. These results will be used throughout the paper in the variational calculations.

\subsection{Notation} \label{Notation subsection}

Given any $x \not= y\in M$, we denote the minimal unit-speed geodesic between them as $\gamma_{x,y}$ and use the parametrization where $\gamma_{x,y}(-\tfrac{d}{2})=x$ and $\gamma_{x,y}(\tfrac d2)=y$ with $d=d(x,y).$  Throughout the paper, we will make the standing assumption that $x$ and $y$ are not cut points, which will be justified by our hypothesis on the diameter of $\Omega$. When unambiguous, will drop the subscripts and simply denote this geodesic by $\gamma$.

We then set \[e_n=\gamma_{x,y}'\left(-\tfrac{d}{2}\right)\in T_{x}\Omega\] to be the unit vector $ T_x M$ pointing along the geodesic (towards $y$). We can extend this to an orthonormal basis $\left \{ e_i \right \}_{i=1,\dots, n}$ at $x$ and then parallel translate this orthonormal basis along $\gamma_{x,y}$ to obtain an orthonormal frame along the geodesic. 
  Setting \[R_{ij}(t)=\langle R(e_i,e_n)e_n,e_j\rangle, \]  we consider the 
matrix-valued solutions $J^{1,0}(t)$ and $J^{0,1}(t)$ which satisfy the system
\begin{align}\label{matrixjacobi-equ}
\begin{cases}
J''(t)+R(t)J(t)=0\\
J^{1,0}(-\tfrac d2)=I_n,\, J^{1,0}(\tfrac d2)=0_n\\
J^{0,1}(-\tfrac d2)=0_n,\, J^{0,1}(\tfrac d2)=I_n.
\end{cases}
\end{align}
Recall that given a geodesic $\gamma: [0,a] \rightarrow M$ whose endpoints are not conjugate, there exists a unique Jacobi field along $\gamma$ for any given boundary values at the endpoints, see \cite[Page 118, Prop. 3.9]{doCarmo}. 
As $y$ is not conjugate to $x$ along $\gamma$, the solutions $J^{1,0}, J^{0,1}$ exist and are unique. Moreover, for all $t \in (-\tfrac d2, \tfrac d2)$ they will be invertible. Furthermore, $(J^{1,0})'(\tfrac{d}{2})$ must also be invertible by the uniqueness of solutions to differential equations.

\subsection{Inhomogeneous Jacobi equations}
Given a smooth path of $n\times n$ matrices $M(t),$ we consider the inhomogeneous equation 
\begin{align}\label{inhomogenousjacobimatrix-equ}
    J''+RJ=M. 
\end{align}
It is possible to solve such an equation in terms of the homogeneous basis solutions $J^{1,0},J^{0,1}$ and a particular solution. We can compute a particular solution by integration. 
\begin{lem}\label{solutionofsecondorderode}
    Suppose that $J$ solves \eqref{inhomogenousjacobimatrix-equ}, then one has that
\begin{align*}
    J(s)&=J^{1,0}(s)J(\tfrac{-d}{2})+J^{0,1}(s)J(\tfrac{d}{2})\\
    &\quad +J^{1,0}(s)\int_{\tfrac{-d}{2}}^sJ^{0,1}(t)[(J^{1,0})'(\tfrac{d}{2})]^{-1}{M(t)}\, dt+J^{0,1}(s)\int_s^{\tfrac{d}{2}}J^{1,0}(t)[(J^{1,0})'(\tfrac{d}{2})]^{-1}M(t)\, dt.
\end{align*}
\end{lem}
 This lemma is analogous to \cite[Lemma 3.2]{figalli2012nearly}, which solved this inhomogeneous problem in terms of initial conditions instead of boundary conditions.
 
\begin{proof} Denote
    \begin{align*}
    F(s)&=J^{1,0}(s)J(\tfrac{-d}{2})+J^{0,1}(s)J(\tfrac{d}{2})\\
    &\quad +J^{1,0}(s)\int_{\tfrac{-d}{2}}^sJ^{0,1}(t)[(J^{1,0})'(\tfrac{d}{2})]^{-1}{M(t)}\, dt+J^{0,1}(s)\int_s^{\tfrac{d}{2}}J^{1,0}(t)[(J^{1,0})'(\tfrac{d}{2})]^{-1}M(t)\, dt.
\end{align*}
Then we see that $F$ coincides with $J$ at the boundary. It thus remains to verify that $F$ solves equation \eqref{inhomogenousjacobimatrix-equ}.
A straightforward computation shows that \begin{align*}
 F''(s) &=-R(s) J^{1,0}(s)J(\tfrac{-d}{2})-R(s) J^{0,1}(s)J(\tfrac{d}{2})\\
 &\quad  -R(s) J^{1,0}(s)\int_{\tfrac{-d}{2}}^sJ^{0,1}(t)[(J^{1,0})'(\tfrac{d}{2})]^{-1}{M(t)}\, dt+(J^{1,0})'(s)J^{0,1}(s)[(J^{1,0})'(\tfrac{d}{2})]^{-1}{M(s)}\\
    &\quad -R(s) J^{0,1}(s)\int_s^{\tfrac{d}{2}}J^{1,0}(t)[(J^{1,0})'(\tfrac{d}{2})]^{-1}{M(t)}\, dt-(J^{0,1})'(s)J^{1,0}(s)[(J^{1,0})'(\tfrac{d}{2})]^{-1}{M(s)}.
    \end{align*}
Hence, 
\begin{align*}
    &F''(s)+R(s)F(s)\\
    &=\Bigl[(J^{1,0})'(s)J^{0,1}(s)-(J^{0,1})'(s)J^{1,0}(s)\Bigr][(J^{1,0})'(\tfrac{d}{2})]^{-1}{M(s)}. 
\end{align*}
However, note that 
\begin{align*}
    \frac{d}{ds}\Bigl[(J^{1,0})'(s)J^{0,1}(s)-(J^{0,1})'(s)J^{1,0}(s)\Bigr]\equiv 0
\end{align*}
and therefore
\begin{align*}
    F''(s)+RF(s)=M(s).
\end{align*}
\end{proof}

\subsection{Two-point Jacobi field comparison}
We now present several comparison results concerning 
Jacobi fields 
\begin{align*}
    J''+\kappa J=0 \quad \textup{in  }(0,l)
\end{align*}
 These are the analogues of the Sturm comparison theorem for Jacobi equations (see, e.g., \cite[Pages 238-239]{doCarmo}) except that we will need to consider boundary value conditions instead of initial value conditions. We first prove a comparison for the solutions of the type $J=J^{1,0}$ or $J=J^{0,1}:$

 \begin{lem}\label{left-endpoint-jacobi-comparison}
 For $i =1,2$, let $J_i$ satisfy 
\begin{align*}
    J_i''+k_iJ_i=0\quad \textup{on} \quad (0,l) 
\end{align*}
with $J_i(0)=1,$ and $J_i(l)=0$.
Suppose that $k_1\leq k_2$ and $J_1,J_2>0$ in $(0,l)$. 
Then 
 \begin{align*}
     J_2\geq J_1 \quad \textup{in }
(0,l). \end{align*}
 \end{lem}
 \begin{proof} The Jacobi field equations imply that \begin{equation}
 J_2''J_1-J_1''J_2+(k_2-k_1)J_1J_2 =0.  \label{JJ-equ}
 \end{equation}
Integrating this equation from $t$ to $l$ gives \begin{align*}
     0&=\int_t^lJ_2''J_1-J_1''J_2+(k_2-k_1)J_1J_2\, ds\\
     &\geq \int_t^l [J_2'J_1-J_1'J_2]'\,ds\\
     &=-J_2'(t)J_1(t)+J_1'(t)J_2(t).
 \end{align*}
In other words, we have that
 \begin{align*}
     \frac{J_2'(t)}{J_2(t)}\geq \frac{J_1'(t)}{J_1(t)}.
 \end{align*}
 Integrating this inequality from $t_0$ to $t$ implies that 
 \begin{align*}
     \log J_2(t)-\log J_2(t_0)\geq  \log J_1(t)-\log J_1(t_0).
 \end{align*}
 Finally, taking the limit as $t_0 \rightarrow 0,$ we obtain the result. 
 \end{proof}

The same result holds (with a nearly identical proof) if we instead assume that $J_i(0)=0,$ and  $J_i(l)=1$.

 This gives the following Jacobi field comparison result, which will be used repeatedly throughout the paper.
 \begin{lem}\label{two-point-jacobi-comparison}
 Let 
\begin{align*}
    J_1''+k_1J_1=0\quad \textup{with }\quad J_1(0)=1,\, J_1(l)=1
\end{align*}and let 
\begin{align*}
    J_2''+k_2J_2=0\quad \textup{with }\quad J_2(0)=1, \, J_2(l)=1.
\end{align*} Suppose that $k_1\leq k_2$ and suppose that $J_1,J_2>0$ in $(0,l),$ then one has that 
 \begin{align*}
     J_2\geq J_1 \quad \textup{in }
(0,l). \end{align*}
 \end{lem}
 \begin{rem}\label{derivative-comparsion-remark}
 This result also provides a comparison for the derivatives at the endpoints. In particular, we find that
 \begin{align}
     J_1'(\tfrac{-d_0}{2})\leq J_2'(\tfrac{-d_0}{2})\quad \textup{as well as }\quad  J_1'(\tfrac{d_0}{2})\geq J_2'(\tfrac{d_0}{2}).  \label{J-derivative-comp} 
 \end{align}
 \end{rem}

Applying the comparison results when the curvature satisfies $\underline \kappa\leq \kappa\leq \overline \kappa$, we have the following estimates.
\begin{prop}\label{J-estimate}
    For $\underline \kappa\leq \kappa\leq \overline \kappa,$ and  $d<\tfrac{\pi}{\sqrt{\overline \kappa}}$ (if $\overline \kappa \leq 0,$ for any $d<\infty),$ one has that for  $s\in[-\tfrac d2, \tfrac d2 ]$
    \begin{eqnarray}      \frac{\textup{cs}_{\underline \kappa}(s)}{\textup{cs}_{\underline \kappa}(\tfrac d2)}&\leq J(s)&\leq\frac{\textup{cs}_{\overline \kappa}(s)}{\textup{cs}_{\overline \kappa}(\tfrac d2)} \\
    \frac{-\textup{sn}_{\underline \kappa}(-\tfrac d2+s)}{\textup{sn}_{\underline \kappa}( d)}&\leq J^{1,0}(s)&\leq\frac{-\textup{sn}_{\overline \kappa}(-\tfrac d2+s)}{\textup{sn}_{\overline \kappa}( d)} \\
     \frac{\textup{sn}_{\underline \kappa}(\tfrac d2+s)}{\textup{sn}_{\underline \kappa}( d)}&\leq J^{0,1}(s)&\leq\frac{\textup{sn}_{\overline \kappa}(\tfrac d2+s)}{\textup{sn}_{\overline \kappa}( d)} 
    \end{eqnarray}
where we denote
\begin{align*}
    \textup{sn}_K(s)=\begin{cases}
        \tfrac{\sin(\sqrt{K}s)}{\sqrt{K}} \quad &\textup{for }K>0\\
        s\quad &\textup{for }K=0\\
        \tfrac{\sinh(\sqrt{-K}s)}{\sqrt{-K}}\quad &\textup{for }K<0,
    \end{cases}
\end{align*}
 $\textup{cs}_{K}(s)=\textup{sn}_K'(s)$ and $\textup{tn}_K(s)=K\tfrac{\textup{sn}_K(s)}{\textup{cs}_K(s)}.$

    \end{prop}

Combining Remark \ref{derivative-comparsion-remark} and Proposition \ref{J-estimate}, one also obtains an estimate on the derivatives at the end points. 
\begin{prop}\label{J-derivative-comparison}
    Under the same assumptions as in Proposition \ref{J-estimate}, one has that 
    \begin{align*}
        \textup{tn}_{\underline \kappa}(\tfrac d2)\leq J'(-\tfrac d2)\leq \textup{tn}_{\overline \kappa}(\tfrac d2) \quad \textup{and }\quad  -\textup{tn}_{\underline \kappa}(\tfrac d2)\geq J'(\tfrac d2)\geq -\textup{tn}_{\overline \kappa}(\tfrac d2)
    \end{align*}
    and that 
    \begin{align*}
       \frac{1}{\textup{sn}_{\underline \kappa }(d)}\leq  (J^{0,1})'(-\tfrac d2)\leq  \frac{1}{\textup{sn}_{\overline \kappa }(d)} \quad \textup{and}\quad \frac{\textup{cs}_{\overline \kappa}(d)}{\textup{sn}_{\overline \kappa}( d)} \leq (J^{0,1})'(\tfrac d2)\leq \frac{\textup{cs}_{\underline \kappa}(d)}{\textup{sn}_{\underline \kappa}( d)}.
    \end{align*}
\end{prop}
From these estimates, we have the following comparison for the quantity $\mathcal C$.
 \begin{lem} \label{TwopointGaussiancomparison}
    A surface $(M^2,g)$ satisfies
    \begin{align}\label{positivityofcurvature}
   \mathcal  C(x,y)\leq -\textup{tn}_{\underline \kappa}(\tfrac{ d(x,y)}{2}) &\iff \kappa\geq \underline \kappa\\
   \mathcal  C(x,y)\geq -\textup{tn}_{\overline \kappa}(\tfrac{ d(x,y)}{2}) &\iff \kappa\leq \overline \kappa. \nonumber
\end{align} 
 \end{lem}
 Note that $2\mathcal C$ is the second order derivative of the distance or index form (see below). This comparison result also follows directly from Hessian (Index) comparison theorem \cite[Page 4]{schoen1994lectures} in the two-dimensional case. See \cite{andrewsclutterbuckmoc} for the higher dimension version.
 
 %
 Recall the index form of a geodesic $\gamma: [0,a] \rightarrow M$ is given by (see, e.g., \cite[Page 17]{Cheeger-Ebin2008})
 \[
 I (V,W) = \int_0^a \langle \nabla_{\gamma'} V,  \nabla_{\gamma'} W \rangle - \langle R(V, \gamma')\gamma', W \rangle,
 \]
 where $V,W$ are vector fields along $\gamma$. It is a symmetric bilinear form and is independent of orientation of $\gamma$. 
 \begin{lem}\label{C-is-derivative-of-distance-lemma}
     Suppose $x,y$ are not cut points of each other. Define $e_2=\gamma_{x,y}'(-\tfrac d2)\in T_xM.$ Choose $e_1\perp e_2$ and parallel translate $e_1$ along $\gamma_{x,y}.$ Set $E_1=e_1(-\tfrac d2)\oplus e_1(\tfrac d2)\in T_xM\bigoplus T_yM$, and $J$ the Jacobi field with $J(\pm \tfrac d2) =e_1(\pm \tfrac d2)$. Then  \begin{align}\label{C-derivative of distance}
         \mathcal C(x,y)=\tfrac 12\nabla^2_{E_1,E_1}d(x,y) = \tfrac 12 I(J,J)
     \end{align}
 \end{lem}
 \begin{proof}
 To compute the derivative of distance,   let $\sigma_1(r)$ be the geodesic with $\sigma_1(0) =x_0, \tfrac{\partial}{\partial r} \sigma_1(0)=e_1(\tfrac{-d_0}{2})$,  $\sigma_2(r)$ be the geodesic with $\sigma_2(0) =y_0, \tfrac{\partial}{\partial r} \sigma_2(0)=e_1(\tfrac{d_0}{2})$, and $\eta(r, s)$, $s \in [-\tfrac{d_0}{2},\tfrac{d_0}{2}]$, be the minimal geodesic connecting $\sigma_1(r)$ and $\sigma_2(r)$.  Since $x,y$ are not cut points of each other, the variation $\eta(r,s)$ is smooth and $\tfrac{\partial \eta}{\partial r} =J$.
     Then by the second variation formula of distance
\small
\begin{equation}\label{second derivative of length of geodesic2}
\begin{aligned}
\frac{\dd^2}{\dd r^2}L[\eta (r,s)]\Big|_{r=0}
&=\int^{\frac{d}{2}}_{-\frac{d}{2}}\Big[ \langle 
\nabla_s \tfrac{\partial \eta}{\partial r}, \nabla_s \tfrac{\partial \eta}{\partial r} \rangle  - \langle R(e_2, \tfrac{\partial \eta}{\partial r})\tfrac{\partial \eta}{\partial r},e_2\rangle \Big]ds + \langle e_2, \nabla_{r}\tfrac{\dd\eta}{\dd r} \rangle \biggr|_{-d/2}^{d/2} \\
&= I(J,J) \\
& = \int^{\frac{d}{2}}_{-\frac{d}{2}}\Big[ \langle 
J',  J \rangle'  - \langle 
J'',  J \rangle  - \langle R(J,e_2) e_2, J\rangle \Big]ds \\
& = \langle  J',  J \rangle \big|_{-\tfrac{d_0}{2}}^{\tfrac{d_0}{2}} \\ 
&=2\mathcal C(x,y).
\end{aligned}
\end{equation}
 \end{proof}
\subsection{An identity for Jacobi fields}
There is one remarkable identity for the derivatives of Jacobi fields at their endpoints. This will play an essential role throughout the proof, so we mention it now.
\begin{lem}\label{HelpforLaplacianTerm}
 Suppose we have two Jacobi fields $J_1$ and $J_2$ which satisfy the following:
 \begin{eqnarray}\label{J1}
    J_1''+kJ_1=0\quad & \textup{with }\quad & J_1(0)=1,\, J_1(l)=0 \\
    \label{J2}
    J_2''+kJ_2=0\quad & \textup{with }\quad & J_2(0)=0, \, J_2(l)=1.
 \end{eqnarray}
 Then we have that 
 \begin{align*}
     J_2'(0)+ J_1'(l)=0.  \end{align*}
 \end{lem}
 \begin{proof}
 Multiplying \eqref{J1} with $J_2$ and \eqref{J2} with $J_1,$ subtracting them from each other and integrating gives that 
 \begin{align*}
     0&=\int_{0}^lJ_1''J_2-J_2''J_1\,ds\\
     &=\Bigl[J_1'(s)J_2(s)-J_2'(s)J_1(s)\Bigr]_{s=0}^{s=l}\\
     &=J_1'(l)+J_2'(0).
 \end{align*}
 \end{proof}
 
 \begin{rem}\label{n-dim-jacobi-symmetry}
     This argument also applies to the matrices $J^{1,0}$ and $J^{0,1}.$ In other words, we have that
     \begin{align*}
        (J^{0,1})'(-d/2)=-(J^{1,0})'(d/2).
     \end{align*}
 \end{rem}

 \section{Modulus of Concavity of the first eigenfunction }\label{MOCsection}

  We first apply the two-point maximum principle in a very general setting -- on a  $n$-dimensional Riemannian manifold for a 
general modulus of concavity function. Then we restrict our attention to two dimensions and choose a specific function to obtain our result. The computations are very long so we postpone some proofs to later sections. 

 To obtain a sufficiently strong concavity estimate of a function $w$ on a strictly convex domain $\Omega \subset M^n$, one considers the two point function
 \begin{equation}\label{Z_with_general_MoC}
     Z(x,y)=\langle \nabla w(y), \gamma_{x,y}'(\tfrac{d}{2})\rangle -\langle \nabla w(x),\gamma_{x,y}'(\tfrac{-d}{2})\rangle -F(x,y). 
  \end{equation}
 and applies a maximum principle to $Z.$ For more background on two-point maximum principles and their applications in geometry, we refer to the surveys by Andrews \cite{andrews2014moduli} and Brendle \cite{brendle2014two}.
 The goal of this section is to construct modulus of concavity for $w$. In other words, we want to find functions $F$ so that $Z\leq 0.$

As discussed in Subsection \ref{sub1.1}, we do not restrict ourselves to radial moduli of concavity and instead consider more general functions $F(x,y)$, which will take to be $C^2$ function and vanish on the diagonal ($x=y$). With these assumptions in place, we first show that $Z$ cannot achieve a positive maximum at the boundary.
 
 \subsection{Boundary Asymptotics} 
Following \cite{andrews2011proof,10.4310/jdg/1559786428},
we show that if $Z$ achieves a positive maximum, it must occur away from the boundary of $\widehat{\Omega} = \Omega\times\Omega - \{(x,x)\ | \ x\in\Omega\}.$ We start by applying \cite[Lemma 3.4]{10.4310/jdg/1559786428}, which states the following.

\begin{lem} \label{Firstboundarylemma}
Let $\Omega$ be a uniformly convex bounded domain in a Riemannian manifold $M^n$, and $u: \bar{\Omega}  \rightarrow \mathbb{R}$ a $C^2$ function which is positive on the interior of $\Omega$ and satisfies $u=0$ on $\partial \Omega$. Furthermore, supposed that $\nabla u \neq 0$ on $\partial \Omega$. There exists $r_1>0$ such that $\left.\nabla^2 \log u\right|_{x}<0$ whenever $d(x, \partial \Omega)<r_1$ and $N \in \mathbb{R}$ such that $\left.\nabla^2 \log u\right|_{x}(v, v) \leq N\|v\|^2$ for all $x \in \Omega$.
\end{lem}

Following \cite[Lemma 3.5]{10.4310/jdg/1559786428}, we also have the following.
\begin{lem} \label{Boundary asymptotics bounded MoC}
    With the assumptions of Lemma \ref{Firstboundarylemma} in place and further assuming that $F$ is continuous on $\overline \Omega\times \overline \Omega,$ for any $\beta >0,$ there exists an open set $U_\beta \subset M\times M$ containing $\partial \widehat \Omega$ such that $Z(x,y)<\beta $ for all $(x,y)\in U_\beta\cap\widehat \Omega.$ 
\end{lem}
 \begin{proof}
     We consider points $(x_0,y_0)\in \partial \widehat\Omega$ and establish the estimate within $B_r(x_0,y_0),$ where $r$ is some small radius (which may depend on $(x_0,y_0)$). The set $U_\beta$ can be taken to be the union of these balls. We distinguish two cases.\\
     \begin{enumerate}
         \item[Case 1.] ($x_0=y_0)$
     For $x,y\in \Omega$ close to $x_0=y_0,$ we have 
     \begin{align*}
        Z(x,y)&=\int_{-\tfrac{d}{2}}^{\tfrac{d}{2}}\Hess w(\gamma_{x,y}'(s),\gamma'_{x,y}(s))\,ds  -F(x,y)\\
        &\leq Nd +F(x_0,y_0)-F(x,y).
     \end{align*}
     We then choose $r$ small such that $Nd(x,y)<\tfrac{\beta}{2}$ and (using the continuity of $F$), $F(x_0,y_0)-F(x,y)<\tfrac{\beta}{2}.$
     \item[Case 2.] ($x_0\neq y_0$) In this case, at least one of $x_0$ and $y_0$ lies on $\partial \Omega.$
Without loss of generality, we take $x_0\in \partial \Omega$. 
  Let $\alpha_0 =\|\nabla u\|(x_0)$, which is strictly positive. By convexity of the domain, $\eta :=\langle -\gamma_{x_0,y_0}'(-\tfrac{d}{2}),\nu(x_0)\rangle >0$ by convexity, where $\nu(x_0)$ is the outward pointing normal. Thus we have the estimate
  \begin{align*}
      -\langle \nabla w(x), \gamma'_{x,y_0}(-\tfrac{d}{2})\rangle =-\frac{1}{u(x)}\langle \nabla u(x), \gamma'_{x,y_0}(-\tfrac{d}{2})\rangle \rightarrow -\infty \quad \textup{as }x\rightarrow x_0,
  \end{align*}
  which implies the claim.
 Note that if $y_0\in \partial \Omega$ as well, we also find that 
 \begin{align*}
     \langle \nabla w(y),\gamma_{x_0,y}'(\tfrac{d}{2})\rangle \rightarrow -\infty \quad \textup{as } y\rightarrow y_0.
 \end{align*}
 Thus we have $Z(x,y)<0$ in a neighborhood of $(x_0,y_0).$
     \end{enumerate}
 \end{proof}

\subsection{Two-point maximum principle on a general manifold}
 
We now assume that $Z$ achieves a
positive maximum $\alpha>0$ in the interior of $\widehat{\Omega}$ at $(x_0,y_0)$. In this case, we have that $x_0 \neq y_0$, and the maximum principle implies that for any $E \in T_{x_0}\Omega \oplus  T_{y_0}\Omega$
\begin{eqnarray}
     0 &=& \nabla_E Z(x_0,y_0),\nonumber \\
     0 &\geq& \nabla ^2_{E,E}Z(x_0,y_0).\nonumber
\end{eqnarray}
   We then construct an orthonormal frame $\{e_i\}$ at $x_0$ such that $e_n = \gamma'_{x_0,y_0}(-\tfrac{d_0}{2})$ and parallel translate it along $\gamma_{x_0,y_0}$ to obtain a local frame along the geodesic.
In these computations, we use the notation $E_i = e_i \oplus e_i \in T_{(x_0,y_0)}\Omega \times \Omega$ for $1 \le i \le n-1$, and $E_n = e_n \oplus (-e_n)$.

Since $\nabla_{e_n} \gamma'_{x_0,y_0}(s) =0$ we immediately have 
\begin{align*}
    \nabla _{E_n,E_n}Z(x_0,y_0)=\langle \nabla _{e_n}\nabla _{e_n}\nabla w(y_0),e_n\rangle -\langle \nabla _{e_n}\nabla _{e_n}\nabla w(x_0),e_n\rangle -\nabla^2_{E_n,E_n}F(x_0,y_0).
\end{align*}

 To compute the derivative $\nabla ^2_{E_i,E_i}Z(x_0,y_0)$
 with $i=1,\dots, n-1$, we introduce variations of $\gamma_{x_0,y_0}(s)$, which are denoted
 \begin{equation} \label{Definition_of_eta}
     \eta_i(r,s): (-\delta, \delta) \times [-\tfrac{d_0}{2}. \tfrac{d_0}{2}] \rightarrow \Omega.
 \end{equation} 

To define this variation, we let $\sigma_1(r)$ be the geodesic with $\sigma_1(0) =x_0, \tfrac{\partial}{\partial r} \sigma_1(0)=e_i(\tfrac{-d_0}{2})$ and  $\sigma_2(r)$ be the geodesic with $\sigma_2(0) =y_0, \tfrac{\partial}{\partial r} \sigma_2(0)=e_i(\tfrac{d_0}{2})$.
We define $\eta_i(r, s)$, $s \in [-\tfrac{d_0}{2},\tfrac{d_0}{2}]$ to be the minimal geodesic connecting $\sigma_1(r)$ and $\sigma_2(r)$.  Since $\Omega$ is strongly convex, the variation $\eta(r,s)$ is smooth.

For fixed $r \neq 0$, the curves $\eta_i(r,\cdot)$ will not be unit speed geodesics in general. Hence we define
   \begin{align*}
    T_i(r,s):=\frac{\eta_i'}{\|\eta_i'\|},
\end{align*}
where we denoted $\partial/\partial s$ by $'$, which is a convention we will use throughout the rest of the paper.

Doing so, we have the following identity.  
\begin{align}  \label{Z_ii1}
   & \nabla^2_{E_i,E_i}Z(x_0,y_0)  =\frac{d^2}{dr^2|_{r=0}}Z(\eta_i(r,\tfrac{-d_0}{2}),\eta_i(r,\tfrac{d_0}{2})) \\
    &=\langle \nabla_{e_i}\nabla _{e_i}\nabla w(y_0),e_n\rangle -\langle \nabla_{e_i}\nabla _{e_i}\nabla w(x_0),e_n\rangle \nonumber\\
   & \quad +2\langle \nabla_{e_i}\nabla w (y_0),\nabla _r T_i(r,\tfrac{d_0}{2})\rangle -2\langle \nabla_{e_i}\nabla w (x_0),\nabla _r T_i(r,\tfrac{-d_0}{2})\rangle  \nonumber  \\
    & \quad+\langle \nabla w(y_0),\nabla _r\nabla _rT_i(0,\tfrac{d_0}{2})\rangle - \langle \nabla w(x_0),\nabla _r\nabla _r T_i(0,\tfrac{-d_0}{2})\rangle-\nabla^2_{E_i,E_i}F(x_0,y_0) \nonumber
    \end{align}

We then denote the variation field 
\begin{equation} \label{Definition of J_1}
  J_i(r,s) =  \tfrac{\partial}{\partial r} \eta_i(r,s),
\end{equation}, which is the Jacobi field along $\eta(s)$ satisfying $J_i(r, -\tfrac{d_0}{2})  =\sigma_1'(r),  \ J_i(r,\tfrac{d_0}{2})  = \sigma_2'(r)$. For simplicity, we will often drop the initial $0$ and denote $J_i(s) = J_i(0,s)$. 

We can simplify this expression using several formulas derived on \cite[Page 363]{10.4310/jdg/1559786428}. In particular, they showed that
\begin{eqnarray*} \label{OnederivativeofT}
    \nabla _r T_i(0,s)&=&-\langle \gamma',J_i'\rangle e_n+J_i', \\
    \nabla _r\nabla _r T_i(0,s)&=&\Bigl(3\langle \gamma',J_i'\rangle^2 -\|J_i'\|^2-\langle \nabla _r\nabla _r \tfrac{\partial \eta_i}{\partial s},e_n\rangle \Bigr)e_n \label{TwoderivativesofT}\\
    & & \quad -2\langle \gamma',J_i'\rangle J_i'+\nabla _r\nabla _r \tfrac{\partial \eta_i}{\partial s}  \nonumber
\end{eqnarray*}
As the tangential component of a Jacobi field is linear and $\langle \gamma',J_i\rangle=0,$ at the end points, we have $\langle \gamma',J_i\rangle=0.$ Therefore $\langle \gamma',J_i'\rangle=0.$
Hence \begin{equation} \label{T-derivative}
   \nabla _r T_i(0,s)=J_i', \ \  \  \nabla _r\nabla _r T_i(0,s)=-\Bigl(\|J_i'\|^2+\langle \nabla _r\nabla _r \tfrac{\partial \eta_i}{\partial s},e_n\rangle \Bigr)e_n+\nabla _r\nabla _r \tfrac{\partial \eta_i}{\partial s}.
\end{equation}
Combining Equations \eqref{T-derivative} and \eqref{Z_ii1}, we have 
\begin{align*}
    & \nabla^2_{E_i,E_i}Z(x_0,y_0) =\\
    &\langle \nabla_{e_i}\nabla _{e_i}\nabla w(y_0),e_n\rangle -\langle \nabla_{e_i}\nabla _{e_i}\nabla w(x_0),e_n\rangle +2\langle \nabla_{e_i}\nabla w (y_0), J_i'(\tfrac{d_0}{2})\rangle \\ 
    &-2\langle \nabla_{e_i}\nabla w (x_0),J_i'(\tfrac{-d_0}{2})\rangle   
    -\|J_i'(\tfrac{d_0}{2})\|^2\langle \nabla w(y_0),e_n\rangle +\|J_i'(\tfrac{-d_0}{2})\|^2 \langle \nabla w(x_0),e_n\rangle\\
    &+\sum_{j=1}^{n-1}\left(\langle \nabla _r\nabla _r \tfrac{\partial \eta_i }{\partial s},e_j\rangle\langle \nabla w(y_0),e_j\rangle-\langle \nabla _r\nabla _r \tfrac{\partial \eta_i }{\partial s},e_j\rangle\langle \nabla w(x_0),e_j\rangle \right) -\nabla^2_{E_i,E_i}F(x_0,y_0).
\end{align*}
We now consider the sum of this expression from $i =1$ to $n$, which yields
\begin{align*}
    0 &\geq \sum_{i=1}^n \nabla^2_{E_i,E_i}Z(x_0,y_0) \\
    &=\langle \Delta \nabla w(y_0),e_n\rangle -\langle \Delta \nabla w(x_0),e_n\rangle)- \sum_{i=1}^n\nabla_{E_i,E_i}^2F(x_0,y_0) \\
    &\quad +\sum_{i=1}^{n-1} \left(2\langle \nabla_{e_i}\nabla w (y_0), J_i'(\tfrac{d_0}{2})\rangle -2\langle \nabla_{e_i}\nabla w (x_0),J_i'(\tfrac{-d_0}{2})\rangle   \right) \\
   & \quad + \sum_{i=1}^{n-1} \left( -\|J_i'(\tfrac{d_0}{2})\|^2\langle \nabla w(y_0),e_n\rangle +\|J_i'(\tfrac{-d_0}{2})\|^2 \langle \nabla w(x_0),e_n\rangle \right) \\
   &\quad+ \sum_{i,j=1}^{n-1} \left(\langle \nabla _r\nabla _r \tfrac{\partial \eta_i }{\partial s},e_j\rangle\langle \nabla w(y_0),e_j\rangle-\langle \nabla _r\nabla _r \tfrac{\partial \eta_i }{\partial s},e_j\rangle\langle \nabla w(x_0),e_j\rangle \right). 
\end{align*}

Our goal now is the simplify this expression as much as possible. The terms in the first line can be simplified using the equation for $\Delta w$ and Bochner's identity: \begin{eqnarray} 
\nonumber -\lambda -\Delta w &= & \|\nabla w\|^2 \label{lap-log} \\
 \Delta\nabla\omega -\Ric(\nabla \omega,\cdot) &= & \nabla\Delta \omega. \label{Bochner}\nonumber
 \end{eqnarray}
Doing so, we find that
\begin{align*}
    0\geq&-e_n(\|\nabla w\|^2)(y_0)+e_n(\|\nabla w\|^2)(x_0)\\
    &\quad - \sum_{i=1}^n \nabla_{E_i,E_i}^2F(x_0,y_0)  +\sum_{i=1}^{n-1} \left(2\langle \nabla_{e_i}\nabla w (y_0), J_i'(\tfrac{d_0}{2})\rangle -2\langle \nabla_{e_i}\nabla w (x_0),J_i'(\tfrac{-d_0}{2})\rangle   \right) \\
   & \quad + \sum_{i=1}^{n-1} \left( -\|J_i'(\tfrac{d_0}{2})\|^2\langle \nabla w(y_0),e_n\rangle +\|J_i'(\tfrac{-d_0}{2})\|^2 \langle \nabla w(x_0),e_n\rangle \right) \\
   & \quad + \Ric (\nabla \omega, e_n)(y_0) - \Ric (\nabla \omega, e_n)(x_0)\\
   &\quad  + \sum_{i,j=1}^{n-1} \left(\langle \nabla _r\nabla _r \tfrac{\partial \eta_i }{\partial s},e_j\rangle\langle \nabla w(y_0),e_j\rangle-\langle \nabla _r\nabla _r \tfrac{\partial \eta_i }{\partial s},e_j\rangle\langle \nabla w(x_0),e_j\rangle \right). 
\end{align*}
The first line can be further simplified:
\begin{eqnarray*}
  -e_n(\nabla w,\nabla w)(y_0)+e_n(\nabla w, \nabla w)(x_0) &= &-2w_n(y_0)w_{nn}(y_0)-2\sum_{j=1}^{n-1} w_j(y_0)w_{jn}(y_0) \\
  & &   +2w_n(x_0)w_{nn}(x_0)  +2\sum_{j=1}^{n-1}w_j(x_0)w_{jn}(x_0),
\end{eqnarray*}
where $w_j = \langle \nabla w, e_j \rangle $ 
 and $\ w_{ij} = \Hess w (e_i, e_j)$. As such, we must compute $w_{jn}$ for $j=1, \dots, n-1$. Since 
\begin{align*}
    0&=\nabla _{e_j\oplus 0}Z(x_0,y_0)\\
    &=\langle \nabla w(y_0), (J_j^{1,0})'(\tfrac{d_0}{2})\rangle -w_{jn}(x_0)-\langle \nabla w(x_0), (J_j^{1,0})'(\tfrac{-d_0}{2})\rangle -\nabla_{e_j\oplus 0}F(x_0,y_0),
\end{align*}
we find that
\begin{align*}
    w_{jn}(x_0)=\langle \nabla w(y_0), (J_j^{1,0})'(\tfrac{d_0}{2})\rangle -\langle \nabla w(x_0), (J_j^{1,0})'(\tfrac{-d_0}{2})\rangle- \nabla_{e_j\oplus 0}F(x_0,y_0),
\end{align*}
where $J_j^{1,0}$ and $J_j^{0,1}$ are the $j$-th components of the matrix-valued Jacobi fields defined by Equation \eqref{matrixjacobi-equ}.
Similarly, using $\nabla_{0\oplus e_j}Z(x_0,y_0) =0,$ we get 
\begin{align*}
    -w_{jn}(y_0)=\langle \nabla w(y_0), (J_j^{0,1})'(\tfrac{d_0}{2})\rangle -\langle \nabla w(x_0), (J_j^{0,1})'(\tfrac{-d_0}{2})\rangle -\nabla _{0\oplus e_j}F(x_0,y_0).
\end{align*}
For the $w_{nn}$ terms, we observe that 
\begin{eqnarray}
\label{w_nn-x}
    0&=&\nabla_{e_n\oplus 0}Z=-w_{nn}(x_0)-\nabla_{e_n\oplus 0}F(x_0,y_0), \\
    0 &= &\nabla _{0\oplus e_n}Z=w_{nn}(y_0)-\nabla _{0\oplus e_n}F(x_0,y_0) \label{w_nn-y}.
\end{eqnarray} 

Using the Jacobi matrix identity from Remark~\ref{n-dim-jacobi-symmetry}, we find 
\begin{align*}
    &2\sum_{j=1}^{n-1} w_j(y_0)\Bigl[\langle \nabla w(y_0), (J_j^{0,1})'(\tfrac{d_0}{2})\rangle -\langle \nabla w(x_0), (J_j^{0,1})'(\tfrac{-d_0}{2})\rangle\Bigr] \\
    &+2\sum_{j=1}^{n-1}w_j(x_0)\Bigl[
     \langle \nabla w(y_0), (J_j^{1,0})'(\tfrac{d_0}{2})\rangle -\langle \nabla w(x_0), (J_j^{1,0})'(\tfrac{-d_0}{2})\rangle 
   \Bigr]\\
   =&2\sum_{i,j=1}^{n-1}w_j(y_0)w_i(y_0)(J^{0,1}_{ji})'(\tfrac{d_0}{2})-w_j(x_0)w_i(x_0)(J^{1,0}_{ji})'(\tfrac{-d_0}{2})\\
   &+2 \sum_{i,j=1}^{n-1}\left(-w_j(y_0)w_i(x_0)(J_{ji}^{0,1})'(\tfrac{-d_0}{2})+w_j(x_0)w_i(y_0)(J^{1,0}_{ji})'(\tfrac{d_0}{2})\right)\\
   =&2\sum_{i,j=1}^{n-1}w_j(y_0)w_i(y_0)(J_{ji})'(\tfrac{d_0}{2})-w_j(x_0)w_i(x_0)(J_{ji})'(\tfrac{-d_0}{2})\\
    &+2 \sum_{i,j=1}^{n-1}(J_{ji}^{0,1})'(\tfrac{-d_0}{2})[w_j(y_0)-w_j(x_0)][w_i(y_0)-w_i(x_0)].
\end{align*}
Noting that \begin{align*}
    &\sum_{i,j=1}^{n-1}(J_{ji}^{0,1})'(\tfrac{-d_0}{2})[w_j(y_0)-w_j(x_0)][w_i(y_0)-w_i(x_0)]\\
    &= \textup{Tr}\left((J^{0,1})'(\tfrac{-d}{2})\circ [\nabla w(y_0)-\nabla w(x_0)]\otimes  [\nabla w(y_0)-\nabla w(x_0)]\right),
\end{align*}
we arrive at the following lemma.

\begin{lem}\label{two-point-differential-inequality}
    
Suppose the function $Z(x,y)$ defined by Equation \ref{Z_with_general_MoC} achieves a positive maximum $\alpha$ on the interior of $\widehat{\Omega}$. Then at that point, the following inequality holds.

\begin{align}  \label{main-two points-n}
     0 \geq & 2\textup{Tr}\Bigl((J^{1,1})'(\tfrac{d_0}{2})\circ \left[\nabla w(y)\otimes \nabla w(y)+\Hess w(y)\right]\Bigr) \\ &-2\textup{Tr} \Bigl ((J^{1,1})'(\tfrac{-d_0}{2})\circ\left[\nabla w(x)\otimes \nabla w(x)+\Hess w(x)\right]\Bigr)  \nonumber\\
     &+2  \textup{Tr}\left((J^{0,1})'(\tfrac{-d}{2})\circ [\nabla w(y)-\nabla w(x)]\otimes  [\nabla w(y)-\nabla w(x)]\right) \nonumber\\
    &+\sum_{i,j=1}^{n-1}\Bigl[\langle \nabla _r\nabla_r\eta'_i(0,\tfrac{d_0}{2}),e_j\rangle -\tfrac{2}{n-1}\nabla _{0\oplus e_j}F(x_0,y_0)+\frac{1}{n-1}\Ric_{jn}(y_0)\Bigr]w_j(y_0) \nonumber\\
    &-\sum_{i,j=1}^{n-1}\Bigl[\langle \nabla _r\nabla_r\eta'_i(0,\tfrac{-d_0}{2}),e_j\rangle +\tfrac{2}{n-1}\nabla _{ e_j\oplus 0}F(x_0,y_0)+\frac{1}{n-1}\Ric_{jn}(x_0)\Bigr]w_j(x_0) \nonumber\\
    &+w_n(y_0)\left(\Ric_{nn}(y_0)-\sum_i\|J_i'(\tfrac{d_0}{2})\|^2-2\nabla_{0\oplus e_n}F\right) \nonumber \\ 
    &-w_n(x_0)\left(\Ric_{nn}(x_0)-\sum_i\|J_i'(-\tfrac{d_0}{2})\|^2+2\nabla_{e_n\oplus 0}F\right)
    -\sum_{i=1}^{n}\nabla^2_{E_i,E_i}F(x_0,y_0). \nonumber
\end{align}
\end{lem}

This can be seen as a two point version of Theorem 2.1 in \cite{surfacepaper1}, and holds for general Riemannian manifolds.

When $M=\mathbb M^n_K$ is a space form, we have \[ ((J^{1,1})_{ij}'(\pm\tfrac{d_0}{2}))_{i,j=1}^{n-1}=\textup{tn}_K(\mp \tfrac{d_0}{2})I_{n-1}.\] 
Hence $\textup{Tr}\Bigl((J^{1,1})' \circ \left(\nabla w\otimes \nabla w+\Hess w\right) \Bigr)$ simplifies in such a way that Equation \eqref{lap-log} can be used. This term seems to be the major obstacle in using the two-point maximum principle for higher dimensional manifolds. This is the first part of the proof where we need two-dimension to proceed with the proof. Moreover the terms $\langle \nabla _r\nabla_r\eta'_i(0,\tfrac{-d_0}{2}),e_j\rangle$ vanish (see \cite{10.4310/jdg/1559786428}, Appendix). When $K=0$, $\nabla _r\nabla_r\eta'_i(0,\tfrac{-d_0}{2})$ is identically zero, not just the vertical components.

\subsection{Specifying to the two-dimensional case}\label{2-d-subsection}
When $n=2$, the trace terms in Equation \eqref{main-two points-n} have only one term each. In this case, $\Ric_p=\kappa(p)g,$ and $J^{1,1}$ can be written in terms of a function $J=J_{x_0,y_0}$ (we omit the subscript in the following computation), which solves the ODE
\begin{align*}
    J''+\kappa J=0, \quad \textup{in }(-\tfrac{d_0}{2},\tfrac{d_0}{2})\quad J(-\tfrac{d_0}{2})=J(\tfrac{d_0}{2})=1
\end{align*}
Inequality \eqref{main-two points-n} becomes 
\begin{align*}
    0\geq &+2J'(\tfrac{d_0}{2})\left(w_{11}(y_0)+w_1^2(y_0)\right) -2J'(\tfrac{-d_0}{2})\left(w_{11}(x_0)+w_{1}^2(x_0)\right) \\
    &\quad 
   +\left(\kappa(y_0) -|J'(\tfrac{d_0}{2})|^2-2\nabla_{0\oplus e_2}F(x_0,y_0)\right)w_2(y_0) \\
   & \quad -\left(\kappa(x_0)-|J'(\tfrac{-d_0}{2})|^2+2\nabla_{e_2\oplus 0}F(x_0,y_0)\right) w_2(x_0)\\
    &\quad+\left(\langle \nabla _r\nabla _r \tfrac{\partial \eta }{\partial s},e_1\rangle-2\nabla_{0\oplus e_1}F(x_0,y_0)\right) w_1(y_0) \\ & \quad -\left(\langle \nabla _r\nabla _r \tfrac{\partial \eta }{\partial s},e_1\rangle +2\nabla_{e_1\oplus 0}F(x_0,y_0)\right)w_1(x_0)  \\
    &\quad +2(J^{0,1})'(\tfrac{-d_0}{2})\Bigl(w_1(y_0)-w_1(x_0)\Bigr)^2-\nabla_{E_1,E_1}^2F(x_0,y_0)-\nabla^2_{E_2,E_2}F(x_0,y_0).
\end{align*}

Note that the first term on the final line is symmetric in $x$ and $y$ via Lemma \ref{HelpforLaplacianTerm}.

Using \eqref{lap-log} and \eqref{w_nn-x}, \eqref{w_nn-y} we have 
\begin{align*}
   & 2J'(\tfrac{d_0}{2})\left(w_{11}(y_0)+w_1^2(y_0)\right) -2J'(\tfrac{-d_0}{2})\left(w_{11}(x_0)+w_{1}^2(x_0)\right)\\
   &=-\lambda \left(2J'(\tfrac{d_0}{2})-2J'(\tfrac{-d_0}{2})\right)-2J'(\tfrac{d_0}{2})\left(w_{22}(y_0)+w_2^2(y_0)\right) +2J'(\tfrac{-d_0}{2})\left(w_{22}(x_0)+w_{2}^2(x_0)\right)\\
   &=-\lambda \left(2J'(\tfrac{d_0}{2})-2J'(\tfrac{-d_0}{2})\right)-2J'(\tfrac{d_0}{2})w_2^2(y_0) +2J'(\tfrac{-d_0}{2})w_{2}^2(x_0)\\
   &\quad -2J'(\tfrac{d_0}{2})\nabla_{0\oplus e_2}F(x_0,y_0)-2J'(\tfrac{-d_0}{2})\nabla_{e_2\oplus 0}F(x_0,y_0).
\end{align*}
We therefore obtain that 
\begin{align}  \label{inequ-max}
    0\geq &-\nabla_{E_1,E_1}^2F(x_0,y_0)-\nabla^2_{E_2,E_2}F(x_0,y_0) -2J'(\tfrac{d_0}{2})\nabla_{0\oplus e_2}F(x_0,y_0)-2J'(\tfrac{-d_0}{2})\nabla_{e_2\oplus 0}F(x_0,y_0) \\
    & \quad +2(J^{0,1})'(\tfrac{-d_0}{2})\Bigl(w_1(y_0)-w_1(x_0)\Bigr)^2 -\lambda \left(2J'(\tfrac{d_0}{2})-2J'(\tfrac{-d_0}{2})\right) \nonumber \\
    &\quad 
   +\left(\kappa(y_0) -|J'(\tfrac{d_0}{2})|^2-2\nabla_{0\oplus e_2}F(x_0,y_0)\right)w_2(y_0) -2J'(\tfrac{d_0}{2})w_2^2(y_0) \nonumber \\
    & \quad -\left(\kappa(x_0)-|J'(\tfrac{-d_0}{2})|^2+2\nabla_{e_2\oplus 0}F(x_0,y_0)\right) w_2(x_0)  +2J'(\tfrac{-d_0}{2})w_{2}^2(x_0) \nonumber \\
    &+\left(\langle \nabla _r\nabla _r \tfrac{\partial \eta }{\partial s},e_1\rangle-2\nabla_{0\oplus e_1}F(x_0,y_0)\right) w_1(y_0)  -\left(\langle \nabla _r\nabla _r \tfrac{\partial \eta }{\partial s},e_1\rangle +2\nabla_{e_1\oplus 0}F(x_0,y_0)\right)w_1(x_0) \nonumber.
\end{align}
\subsection{Specifying to a special function \textit{F}}
As we mentioned previously, the term  $\langle \nabla _r\nabla _r \tfrac{\partial \eta }{\partial s},e_1\rangle$ vanishes in constant curvature, so the function $F$ can be chosen to a radial function (i.e. only depending on the distance between the two points). 
When the curvature is not constant, we will need to introduce a component of $F$ to deal with the term $\langle \nabla _r\nabla _r \tfrac{\partial \eta }{\partial s},e_1\rangle$. Quite surprising, it turns out that there is a particular function which allows this to be done. 

Let $F$ be given by 
    \begin{align*}
        F(x,y)= 2\phi(\tfrac{d(x,y)}{2})+\mathcal C(x,y)
    \end{align*}
    for some $C^2$ function $\phi,$ where $\mathcal C$ is defined in \eqref{def-of-C}.
    
 The key property which makes this function useful is the following proposition.
\begin{prop}\label{CancellationProposition} For $x,y\in \Omega,$ one has that 
\begin{align}\label{equality1inkeyprop}
        \langle \nabla _r\nabla _r \tfrac{\partial \eta }{\partial s}(0,\tfrac{d_0}{2}),e_1\rangle-2\nabla_{0\oplus e_1}\mathcal C(x_0,y_0)=\langle \nabla _r\nabla _r \tfrac{\partial \eta }{\partial s}(0,\tfrac{-d_0}{2}),e_1\rangle +2\nabla_{e_1\oplus 0}\mathcal C(x_0,y_0)=0.
    \end{align}
\end{prop}

The proof of this proposition can be found by computing the relevant Jacobi fields, and will be presented in Section~\ref{derivativesection}. Note that in view of \eqref{equality1inkeyprop}, the fifth line of \eqref{inequ-max} is zero since by the first variation formula the derivatives of the distance are zero. 

A second difficulty is that, since the curvature is non constant, we cannot combine the $w_2$ terms in the third and fourth lines. We will deal with this issue below (see Lemma \ref{e_2-derivative-of-C} and thereafter).

We now reduce Inequality \eqref{inequ-max} to a simpler form, throughout these computations we write $d=d(x,y)$ and $d_0=d(x_0,y_0).$ Furthermore, we will often drop the argument from $\textup{tn}_{\underline \kappa}(\tfrac{d_0}{2})$ and simply denote this quantity by $\textup{tn}_{\underline \kappa}$ (and similarly for $\textup{tn}_{\overline \kappa}$).

To control the terms in the third and fourth line, we make use of two properties. First, we use the following lemma (whose proof will be done in Section \ref{derivativesection}). \begin{lem}\label{e_2-derivative-of-C} We have 
  \begin{align*}
      2 \nabla _{e_2\oplus 0}\mathcal C&= \kappa(x)+|J'_{x,y}(-\tfrac d2)|^2\\
 2 \nabla _{0\oplus e_2}\mathcal C&=-\kappa(y)-|J'_{x,y}(\tfrac d2)|^2.
  \end{align*}
\end{lem}

Second, we assume that $u$ is log-concave which implies \begin{align}\label{log-concavity}
w_2(y_0)\leq w_2(x_0). 
\end{align}
This might seem to be circular, since the entire goal of this argument is to establish a log-concavity estimate for $u$. However, this assumption will later be removed by a continuity method (see proof of Proposition \ref{Propositiongeneralmoc} below).

Using Lemma \ref{e_2-derivative-of-C}, we consider the terms 
\begin{align*}
    &\left(\kappa(y_0) -|J'(\tfrac{d_0}{2})|^2-2\nabla_{0\oplus e_2}F(x_0,y_0)\right)w_2(y_0) \\
    & -\left(\kappa(x_0)-|J'(\tfrac{-d_0}{2})|^2+2\nabla_{e_2\oplus 0}F(x_0,y_0)\right) w_2(x_0)\\
    &=-2\left(w_2(y_0)-w_2(x_0)\right)\phi'(\tfrac {d_0}{2}) \\ &  +2\kappa(y_0) w_2(y_0)  -2\kappa(x_0) w_2(x_0).
\end{align*}
Using the assumption $w_2(y_0)\leq w_2(x_0)$, we break the analysis of these terms into three cases. 
 \begin{enumerate}
     \item[Case 1.]  $(0>w_2(x_0)>w_2(y_0))$ 
In this case, we have that
\begin{align*}
    &2\kappa(y_0) w_2(y_0)  -2\kappa(x_0) w_2(x_0) 
    \\
    &=2\kappa(y_0)\left(w_2(y_0)-w_2(x_0)\right)+\varepsilon(x_0,y_0)w_2(x_0), 
\end{align*}
where 
\begin{align}\label{epsilon-case-1}
\varepsilon(x,y)=2\left(\kappa(y_0)-\kappa(x_0)\right).
\end{align}
Since $J'(\tfrac{d_0}{2}) <0$, the assumption implies the estimate
\begin{align*}
    -2J'(\tfrac{d_0}{2})w_2^2(y_0)\geq -2J'(\tfrac{d_0}{2})\left(w_2(y_0)-w_2(x_0)\right)^2\geq  2\textup{tn}_{\underline \kappa}(\tfrac{d_0}{2})\left(w_2(y_0)-w_2(x_0)\right)^2, 
\end{align*}
where we have used Proposition \ref{J-derivative-comparison} in the last inequality. Furthermore, when completing the square, we find that 
\begin{align*}
    2J'(\tfrac{-d_0}{2})w_2^2(x_0)+\varepsilon(x_0,y_0)w_2(x_0)=2J'(\tfrac{-d_0}{2})\Bigl(w_2(x_0)+\frac{\varepsilon(x_0,y_0)}{4J'(\tfrac{-d_0}{2})}\Bigr)^2-\frac{\varepsilon^2(x_0,y_0)}{8J'(\tfrac{-d_0}{2})}. 
\end{align*}
Combining this equation, using Lemma \eqref{e_2-derivative-of-C}, and Proposition \ref{J-derivative-comparison}, together with the log-concavity assumption \eqref{log-concavity}, \eqref{inequ-max} becomes the following.  
\begin{align}\label{Ineq-case-1}
    0\geq &-\nabla_{E_1,E_1}^2F(x_0,y_0)-\nabla^2_{E_2,E_2}F(x_0,y_0) -2J'(\tfrac{d_0}{2})\nabla_{0\oplus e_2}F(x_0,y_0)-2J'(\tfrac{-d_0}{2})\nabla_{e_2\oplus 0}F(x_0,y_0) \\
    & \quad +2(J^{0,1})'(\tfrac{-d_0}{2})\Bigl(w_1(y_0)-w_1(x_0)\Bigr)^2 -\lambda \left(2J'(\tfrac{d_0}{2})-2J'(\tfrac{-d_0}{2})\right)-2\left(w_2(y_0)-w_2(x_0)\right)\phi'(\tfrac {d_0}{2}) \nonumber \\
&\quad +2\kappa(y_0)\left(w_2(y_0)-w_2(x_0)\right)+2\textup{tn}_{\underline \kappa}(\tfrac{d_0}{2})\left(w_2(y_0)-w_2(x_0)\right)^2-\frac{\varepsilon^2(x_0,y_0)}{8\textup{tn}_{\underline \kappa}(\tfrac {d_0}{2})}.
 \nonumber 
\end{align}
\item[Case 2.] $(w_2(x_0)\geq w_2(y_0)>0)$
 
 In this case, we observe that 
\begin{align*}
    &2\kappa(y_0) w_2(y_0)  -2\kappa(x_0) w_2(x_0)  
    \\
    &=2\kappa(x_0)\left(w_2(y_0)-w_2(x_0)\right)+\varepsilon(x_0,y_0)w_2(y_0) 
    .
\end{align*}
 Since $J'(-\tfrac{d_0}{2}) >0$, the standing assumption implies the estimate
\begin{align*}
    2J'(-\tfrac{d_0}{2})w_2^2(x_0)\geq 2J'(-\tfrac{d_0}{2})\left(w_2(y_0)-w_2(x_0)\right)^2\geq 2  \textup{tn}_{\underline \kappa}(\tfrac{d_0}{2})\left(w_2(y_0)-w_2(x_0)\right)^2, 
\end{align*}
where in the last inequality we have again used Proposition \ref{J-derivative-comparison}. 

When completing the square, we get that 
\begin{align*}
    -2J'(\tfrac{d_0}{2})w_2^2(y_0)+\varepsilon(x_0,y_0)w_2(y_0)=-2J'(\tfrac{d_0}{2})\Bigl(w_2(y_0)-\frac{\varepsilon(x_0,y_0)}{4J'(\tfrac{d_0}{2})}\Bigr)^2+\frac{\varepsilon^2(x_0,y_0)}{8J'(\tfrac {d_0}{2})}. 
\end{align*}
Similarly to Case 1, we obtain the following inequality:
\begin{align}\label{inequ-case-2}
    0\geq &-\nabla_{E_1,E_1}^2F(x_0,y_0)-\nabla^2_{E_2,E_2}F(x_0,y_0) -2J'(\tfrac{d_0}{2})\nabla_{0\oplus e_2}F(x_0,y_0)-2J'(\tfrac{-d_0}{2})\nabla_{e_2\oplus 0}F(x_0,y_0) \\
    & \quad +2(J^{0,1})'(\tfrac{-d_0}{2})\Bigl(w_1(y_0)-w_1(x_0)\Bigr)^2 -\lambda \left(2J'(\tfrac{d_0}{2})-2J'(\tfrac{-d_0}{2})\right)-2\left(w_2(y_0)-w_2(x_0)\right)\phi'(\tfrac {d_0}{2}) \nonumber \\
&\quad +2\kappa(x_0)\left(w_2(y_0)-w_2(x_0)\right)+2\textup{tn}_{\underline \kappa}(\tfrac{d_0}{2})\left(w_2(y_0)-w_2(x_0)\right)^2+\frac{\varepsilon^2(x_0,y_0)}{8\textup{tn}_{\underline \kappa}(\tfrac {d_0}{2})}
 \nonumber 
\end{align}

Note that Inequality \eqref{inequ-case-2} differs from Inequality \eqref{Ineq-case-1} only in that several of the terms are evaluated at $x_0$ rather than $y_0$.

\item[Case 3.] $(w_2(x_0)\geq 0\geq  w_2(y_0))$ 

In this case, we have
\begin{align*}
    &2\kappa(y_0) w_2(y_0)  -2\kappa(x_0) w_2(x_0)  
    \\
    &\geq 2\overline\kappa\left(w_2(y_0)-w_2(x_0)\right) 
\end{align*}

Moreover, by Proposition \ref{J-derivative-comparison} \begin{align*}
    J'(-\tfrac {d_0}{2})w_2^2(x_0)-2J'(\tfrac{d_0}{2})w_2^2(y_0)&\geq 2\textup{tn}_{\underline \kappa}(\tfrac {d_0}{2})\left(w_2^2(x_0)+w_2^2(y_0)\right)\\
    &\geq \textup{tn}_{\underline \kappa}(\tfrac {d_0}{2})\left(w_2(x_0)-w_2(y_0)\right)^2.\nonumber
\end{align*}
 \end{enumerate}

In all three cases, we find the following inequality
\begin{align}\label{ineq-max-after-case}
0\geq    &-\nabla_{E_1,E_1}^2F(x_0,y_0)-\nabla^2_{E_2,E_2}F(x_0,y_0) -2J'(\tfrac{d_0}{2})\nabla_{0\oplus e_2}F(x_0,y_0)-2J'(\tfrac{-d_0}{2})\nabla_{e_2\oplus 0}F(x_0,y_0) \\
    & \quad +2(J^{0,1})'(\tfrac{-d_0}{2})\Bigl(w_1(y_0)-w_1(x_0)\Bigr)^2 -4\lambda \mathcal C(x_0,y_0)-2\left(w_2(y_0)-w_2(x_0)\right)\phi'(\tfrac {d_0}{2}) \nonumber \\
&\quad +2\overline\kappa\left(w_2(y_0)-w_2(x_0)\right)+\textup{tn}_{\underline \kappa}(\tfrac{d_0}{2})\left(w_2(y_0)-w_2(x_0)\right)^2+\frac{\varepsilon^2(x_0,y_0)}{8\textup{tn}_{\underline \kappa}(\tfrac {d_0}{2})}\nonumber
\end{align}
where we have rewritten $\left(2J'(\tfrac{d_0}{2})-2J'(\tfrac{-d_0}{2})\right)=4\mathcal C(x_0,y_0).$

We now estimate all the terms of \eqref{ineq-max-after-case}, in terms of $\phi$ and $\mathcal C.$\\

Since we have assumed that $Z(x_0, y_0) = \alpha$, we have that 
\begin{align*}
    w_2(y_0)-w_2(x_0)=F(x_0,y_0)+\alpha=2\phi(\tfrac{d_0}{2})+\mathcal C(x_0,y_0)+\alpha. 
\end{align*}
This gives 
\begin{eqnarray} \label{phi'-terms} & & \\
    -2\left(w_2(y_0)-w_2(x_0)\right)\phi'(\tfrac {d_0}{2}) &=&-2\alpha\phi'(\tfrac{d_0}{2})-2\phi'(\tfrac{d_0}{2})\mathcal C(x_0,y_0)  -4\phi'(\tfrac{d_0}{2})\phi(\tfrac{d_0}{2}), \nonumber   \\ 
\label{w_2-terms} & & \\
   2\overline \kappa\left(w_2(y_0)-w_2(x_0)\right)
   &=&2\overline \kappa\mathcal C(x_0,y_0)
   +4\overline \kappa\phi(\tfrac{d_0}{2})+2\alpha \overline \kappa , \nonumber \\
    \label{w_2-squared-terms} & & \\
 \textup{tn}_{\underline\kappa}(\tfrac{d_0}{2})\left( w_2(y_0)-w_2(x_0)\right)^2 \nonumber
   &=&\alpha \textup{tn}_{\underline\kappa}(\tfrac{d_0}{2})\left(\alpha +2\mathcal C(x_0,y_0)+4\phi(\tfrac{d_0}{2})\right) \nonumber + 4\phi^2(\tfrac{d_0}{2})\textup{tn}_{\underline \kappa}(\tfrac{d_0}{2}) \\
   & &  +2\mathcal C(x_0,y_0)\Bigl( 2 \textup{tn}_{\underline \kappa}(\tfrac{d_0}{2})\phi(\tfrac{d_0}{2}) +\frac{1}{2} \textup{tn}_{\underline \kappa}(\tfrac{d_0}{2})\mathcal C(x_0,y_0)\Bigr). \nonumber
\end{eqnarray}
From  Lemma \ref{C-is-derivative-of-distance-lemma}, we have that 
\begin{eqnarray}
\label{E_1-derivative-of-F}
    -\nabla^2_{E_1,E_1}F(x_0,y_0)&=&-2\mathcal C(x_0,y_0)\phi '(\tfrac{d_0}{2})+\nabla _{E_1,E_1}^2 (-\mathcal C)(x_0,y_0). \\
\label{E_2-derivative-of-F}-\nabla_{E_2,E_2}^2F(x_0,y_0)&=&\nabla^2_{E_2,E_2}(-\mathcal C)(x_0,y_0) -2\phi ''(\tfrac{d_0}{2}).
\end{eqnarray}
 Also by Lemma \ref{e_2-derivative-of-C}
\begin{align}\label{0opluse_2-terms-of-F}
 &-2J'(\tfrac{d_0}{2})\nabla_{0\oplus e_2}F(x_0,y_0)-2J'(\tfrac{-d_0}{2})\nabla_{e_2\oplus 0}F(x_0,y_0)\\ \nonumber
 &=-2\left(J'(\tfrac{d_0}{2})-J'(-\tfrac{d_0}{2})\right)\phi'(\tfrac{d_0}{2})+J'(\tfrac{d_0}{2})\kappa(y_0)+J'(\tfrac {d_0}{2})^3-J'(-\tfrac{d_0}{2})\kappa(x_0)-J'(-\tfrac {d_0}{2})^3.
\end{align}
Finally, since
\begin{align}\label{extra-positivity-term}
2(J^{0,1})'(\tfrac{-d_0}{2})\Bigl(w_1(y_0)-w_1(x_0)\Bigr)^2\geq 0,
\end{align} this term can be dropped from the analysis. 
Thus, combining \eqref{phi'-terms}, \eqref{w_2-terms},\eqref{w_2-squared-terms}, \eqref{E_1-derivative-of-F}, \eqref{E_2-derivative-of-F}, \eqref{0opluse_2-terms-of-F}, and \eqref{extra-positivity-term} into \eqref{ineq-max-after-case} gives
\begin{align}\label{ineq-max-after-all}
0\geq    &-2\phi ''(\tfrac{d_0}{2})-4\phi'(\tfrac{d_0}{2})\phi(\tfrac{d_0}{2})-8\phi'(\tfrac{d_0}{2})\mathcal C(x_0,y_0)+4\overline \kappa\phi(\tfrac{d_0}{2})\\ \nonumber
&\quad + 4\phi^2(\tfrac{d_0}{2})\textup{tn}_{\underline \kappa}(\tfrac{d_0}{2})+ 4\mathcal C(x_0,y_0) \textup{tn}_{\underline \kappa}(\tfrac{d_0}{2})\phi(\tfrac{d_0}{2})\\\nonumber
&\quad+\alpha\Bigl(2 \overline \kappa-2\phi'(\tfrac{d_0}{2})+\textup{tn}_{\underline\kappa}(\tfrac{d_0}{2})\left(\alpha +2\mathcal C(x_0,y_0)+4\phi(\tfrac{d_0}{2})\right)\Bigr)
   \\\nonumber
   &\quad +2\overline \kappa\mathcal C(x_0,y_0)  +\textup{tn}_{\underline \kappa}(\tfrac{d_0}{2})\mathcal C^2(x_0,y_0)-4\lambda \mathcal C(x_0,y_0)\\\nonumber
   &\quad+J'(\tfrac{d_0}{2})\kappa(y_0)+[J'(\tfrac {d_0}{2})]^3-J'(-\tfrac{d_0}{2})\kappa(x_0)-[J'(-\tfrac {d_0}{2})]^3\\
\nonumber &\quad +\nabla _{E_1,E_1}^2 (-\mathcal C)(x_0,y_0)+\nabla^2_{E_2,E_2}(-\mathcal C)(x_0,y_0)-\frac{\varepsilon^2(x_0,y_0)}{8\textup{tn}_{\underline \kappa}(\tfrac {d_0}{2})}.
\end{align}
Computing the term $\nabla ^2_{E_1,E_1}(-\mathcal C)+\nabla^2_{E_2,E_2}(-\mathcal C)$ see Proposition \ref{E_1+E_2-derivatives-of-C}, we get 
\begin{align*}
    \left(\nabla ^2_{E_1,E_1}+\nabla_{E_2,E_2}^2\right)(-\mathcal C)&
   =\kappa(x_0)J'(\tfrac{-d_0}{2})-\kappa(y_0)J'(\tfrac{d_0}{2})+[J'(\tfrac{-d_0}{2})]^3-[J'(\tfrac{d_0}{2})]^3\\
    \quad &+\mathcal C(x_0,y_0)\nabla_{e_2\oplus (-e_2)}\mathcal C(x_0,y_0)+ \mathcal D(x_0,y_0),
\end{align*}
where $\mathcal D$ is defined in Equation \eqref{def-of-D} and satisfies $|\mathcal D(x,y)|\leq Cd(x,y)\left(|\nabla \kappa|_{\infty}+|\Delta \kappa|_\infty\right) $ for some $C>0$ (to be made precise in Section \ref{DerivativesofCSection}). 
This cancels the fifth line of \eqref{ineq-max-after-all}. Estimating \begin{align*}
    \mathcal C(x_0,y_0)\nabla_{e_2\oplus (-e_2)}\mathcal C(x_0,y_0)\geq \mathcal C(x_0,y_0)\left(\overline \kappa+\textup{tn}_{\overline \kappa}^2\right),
\end{align*}
 we arrive at the following theorem.

 \begin{thm}\label{Propositiongeneralmoc}
 Suppose $(M^2,g)$ is a complete manifold of positive curvature. Let $\Omega \subset M^2$ be a uniformly convex domain with $w=\log u$ the log of the first Dirichlet eigenfunction. Suppose further that $\phi:[0,D/2]\rightarrow \mathbb R$ is a $C^2$ function with $\phi(0)=0$ which satisfies the following inequalities (for any $x,y\in \Omega$):
 \begin{align}\label{inequalities-from-proposition}
     \begin{cases}
         0\leq & -2\phi ''-4\phi \phi'+4\overline\kappa\phi+4\textup{tn}_{\underline \kappa}(\tfrac{d(x,y)}{2})\phi^2+(-4\phi' +2 \textup{tn}_{\underline \kappa}(\tfrac{d(x,y)}{2})\phi)2\mathcal C(x,y)\\ 
   & +\Bigr(-4\lambda+3\overline\kappa+\textup{tn}_{\overline\kappa}^2(\tfrac {d(x,y)}{2})+ \textup{tn}_{\underline \kappa}(\tfrac{d(x,y)}{2})\mathcal C(x,y)\Bigl ) \mathcal C(x,y)\\
   &-\frac{\varepsilon^2(x,y)}{8\textup{tn}_{\underline \kappa}(\tfrac{d(x,y)}{2})}+\mathcal D(x,y)\\
    0 \leq &2 \overline \kappa-2\phi'(\tfrac{d(x,y)}{2})+\textup{tn}_{\underline\kappa}(\tfrac{d(x,y)}{2})\left(2\mathcal C(x,y)+4\phi(\tfrac{d(x,y)}{2})\right)\\
    0>&\phi'(0)-\frac{\underline \kappa}{2}
     \end{cases},
  \end{align}
  where $\varepsilon$ is defined \eqref{epsilon-case-1} and $\mathcal D$ is defined in \eqref{def-of-D}.

Then, for any $x,y\in \Omega$
 \begin{equation} \label{MOC_estimate}
     \langle \nabla w(y), \gamma_{x,y}'(\tfrac{d}{2})\rangle - \langle \nabla w(x), \gamma_{x,y}'(\tfrac{-d}{2})\rangle \leq 2\phi(\tfrac{d}{2})+\mathcal C(x,y)
 \end{equation}
\end{thm}

  \begin{proof}  

We consider an increasing family of convex domains $\Omega(t) \subset \Omega$ such that $\Omega(0)$ is a sufficiently close to a very small ball and where $\Omega (1)=\Omega$. Such a family can be constructed by deforming the boundary $\partial \Omega$ via curve shortening flow. This process gives a decreasing family of convex domains which converge to a round point at the time of singularity \cite{Gage1990}. As such, we can stop the flow close to the singular time (when the domain is very close to a small ball). Doing this process backwards in time (and reparametrizing so that it occurs in between $t=0$ and $t=1$), we have the family of interest.

 Since $\Omega(0)$ is very close to a small ball (and the metric is nearly Euclidean in a small neighborhood), elliptic regularity implies that for $t=0$ (and for $t$ sufficiently small), the first eigenfunction of $\Omega(t)$ is strictly log-concave.

 Starting at $\Omega(0)$, we move forwards in time until the first time $t_0$ where $\textup{max}\Hess w=0,$ where the maximum ranges over the unit tangent bundle. By the maximum principle computation above, $Z$ cannot achieve a positive maximum in the interior and thus $Z$ must be non-positive (since it is non-positive on the boundary). However, for any unit vector $X\in T_p\Omega$ and geodesic $\gamma$ such that $\gamma'(0)=X$, we have the inequality
\begin{align*}
    \Hess w(X,X) &=\lim_{d\rightarrow 0^+}\frac{1}{d}\int_{-\tfrac{d}{2}}^{\tfrac{d}{2}}\Hess w (\gamma'(s),\gamma'(s))\,ds\\
    &\leq \lim_{d\rightarrow 0^+}\frac{1}{d}\left(2\phi(\tfrac d2)+\mathcal C(\gamma(\tfrac{-d}{2}),\gamma(\tfrac{d}{2}))\right)\\
    &=\phi'(0)-\frac{\kappa(p)}{2}<0,
\end{align*}
 which induces a contradiction.
 
 This shows that the eigenfunction is log-concave for all $t \in [0,1]$. As such, we can apply the maximum principle to show that $Z$ must be non-positive at $t=1.$ In other words,
\begin{align*}
    \langle \nabla w(y), \gamma'_{x,y}(\tfrac{d}{2})\rangle -\langle \nabla w(x),\gamma'_{x,y}(-\tfrac{d}{2})\rangle \leq  2\phi(\tfrac{d}{2})+\mathcal C(x,y).
\end{align*}
  \end{proof}
  
In Section \ref{DerivativesofCSection}, we estimate $\mathcal D$ and in particular obtain the following estimate to deal with all non-ODE terms in \eqref{inequalities-from-proposition}.
\begin{lem} \label{estiamtesonderivative}
    When the diameter satisfies $D\leq \pi/(2\sqrt{\overline \kappa})$, we have the estimate
 \begin{align}
\nonumber&\Bigl(-4\lambda+3\overline\kappa+\textup{tn}_{\overline\kappa}^2(\tfrac {d(x,y)}{2})+ \textup{tn}_{\underline \kappa}(\tfrac{d(x,y)}{2})\mathcal C(x,y) \Bigr)\mathcal C(x,y)\\ \nonumber
   &-\frac{\varepsilon^2(x,y)}{8\textup{tn}_{\underline \kappa}(\tfrac{d(x,y)}{2})}+\mathcal D(x,y)\\
   \nonumber
   \geq& \Bigr(-4\lambda+3\overline\kappa+ (\overline \kappa-\underline \kappa) (1+\frac\pi2)+\sqrt 2 \overline \kappa \left(\frac{\overline \kappa-\underline \kappa}{\underline \kappa}\right)^2\left(\frac{4{\overline \kappa}}{ \pi\sqrt{\underline \kappa}}+1\right)^2\Bigl )\mathcal C(x,y), \\
   &\,+2\mathcal C(x,y)\left(\frac{-2(\inf\Delta \kappa)_{-}}{\underline \kappa}+\left(\frac{1}{\underline \kappa^2}+\frac{3\pi^2\sqrt8 }{4\overline\kappa\underline \kappa}\right)|\nabla \kappa|^2_\infty+\frac{4\pi|\nabla \kappa|_\infty }{\underline\kappa}\right).\nonumber
 \end{align}
   \end{lem} 
 
\subsection{Explicit modulus of concavity estimates}

We are now in the position to starting applying Theorem \ref{Propositiongeneralmoc} to several choices of $\phi.$ We start with the simplest choice, which is $\phi \equiv 0.$
\begin{thm}\label{nonradialthm}
Suppose that $(M^2,g)$ is a surface of positive curvature $\kappa$ such that $0<\underline \kappa\leq \kappa <\overline \kappa.$ Let $\Omega \subset M^2$ be uniformly and geodesically convex with $\textup{diam}(\Omega)\leq \tfrac{\pi}{2\sqrt{\overline \kappa}}$ and suppose that 
\begin{eqnarray}\label{assumption-thm-1}
-2\lambda+\tfrac32\overline\kappa+ (\overline \kappa-\underline \kappa) (\tfrac12+\tfrac\pi4)+\sqrt 2 \overline \kappa \left(\tfrac{\overline \kappa-\underline \kappa}{\underline \kappa}\right)^2\left(\tfrac{4{\overline \kappa}}{ \pi\sqrt{\underline \kappa}}+1\right)^2 \\
-\tfrac{2}{\underline \kappa}(\inf\Delta \kappa)_{-}+\left(\tfrac{1}{\underline \kappa^2}+\tfrac{3\pi^2\sqrt8 }{4\overline\kappa\underline \kappa}\right)|\nabla \kappa|^2_\infty+\tfrac{4\pi }{\underline\kappa}|\nabla \kappa|_\infty<0. \nonumber
   \end{eqnarray}
Then one has that 
\begin{align}\label{two-point-moc}
     \langle \nabla w(y), \gamma_{x,y}'(\tfrac{d}{2})\rangle - \langle \nabla w(x), \gamma_{x,y}'(-\tfrac{d}{2})\rangle \leq \mathcal C(x,y).
\end{align}
\end{thm}
\begin{rem}
Theorem \ref{nonradialthm} also holds true under the weaker assumption that without any diameter restriction on $\Omega$ and any $x,y\in\Omega$
\begin{align}\label{weak-assumption-thm-1}
\Bigr(-4\lambda+3\overline\kappa+\textup{tn}_{\overline\kappa}^2(\tfrac {d}{2})+ \textup{tn}_{\underline \kappa}(\tfrac{d}{2})\mathcal C(x,y)\Bigl ) \mathcal C(x,y)
   -\frac{\varepsilon^2(x,y)}{8\textup{tn}_{\underline \kappa}(\tfrac{d}{2})}+\mathcal D(x,y)>0,
    \end{align}
    where $\mathcal C$ is defined in \eqref{def-of-C}, $\mathcal D$ in \eqref{def-of-D} and $\varepsilon$ in \eqref{epsilon-case-1}. 
\end{rem}
\begin{proof}[Proof of Theorem \ref{nonradialthm}]
    This follows from Proposition~\ref{Propositiongeneralmoc} with $\phi \equiv 0$, Estimate \eqref{estiamtesonderivative} and 
    the fact that $\mathcal C(x,y) <0.$ 
\end{proof}
\begin{rem}
Note that if $(M^2,g)$ satisfies the assumptions of Theorem \ref{nonradialthm}, one recovers the Hessian estimate from \cite{surfacepaper1}. Indeed, dividing \eqref{two-point-moc} by $d$ and after passing to the limit as $d\rightarrow 0^+,$ one recovers 
\begin{align}\label{one-point-estimate}
   \Hess w\leq -\frac{\kappa}{2}.
\end{align}
Note that in that paper, we used the function $\frac{\kappa}{2}$ to absorb problematic term in a one-point maximum principle. Interestingly, $\mathcal C$ helps us to control problematic terms in a similar way and these two quantities are related by Equation \eqref{collapse-of-C}. 
\end{rem}
We now apply Theorem \ref{MOC-thm} again, this time with $\phi=\psi+\textup{tn}_{\underline \kappa}$ 
and derive differential inequalities for $\psi$.
The result below generalizes Theorem 3.6 in \cite{10.4310/jdg/1559786428} to surfaces of non-constant sectional curvature. Their version of the theorem provides a parabolic version of this argument. We only state the elliptic version as the proof of the parabolic version is almost identical.
 \begin{thm}\label{MOC-thm}
 Suppose the $(M^2,g)$ is a complete manifold of positive sectional curvature $\kappa$, such that $0<\underline \kappa\leq \kappa\leq \overline \kappa$ which satisfies
 \begin{align}\label{curvature-assumption-in-moc-thm}
&-\tfrac12\overline\kappa+ (\overline \kappa-\underline \kappa) (\tfrac12+\tfrac\pi4)+\sqrt 2 \overline \kappa \left(\tfrac{\overline \kappa-\underline \kappa}{\underline \kappa}\right)^2\left(\tfrac{4{\overline \kappa}}{ \pi\sqrt{\underline \kappa}}+1\right)^2\\\nonumber
&\quad -\tfrac{2}{\underline \kappa}(\inf\Delta \kappa)_{-}+\left(\tfrac{1}{\underline \kappa^2}+\tfrac{3\pi^2\sqrt8 }{4\overline\kappa\underline \kappa}\right)|\nabla \kappa|^2_\infty+\tfrac{4\pi }{\underline\kappa}|\nabla \kappa|_\infty<0
\end{align}
Then we have that for any convex $\Omega\subset M,$ with $\textup{diam}(\Omega)=D\leq \tfrac{\pi}{2\sqrt{\overline \kappa}}$ 
 \begin{align*}
     \langle \nabla w(y), \gamma_{x,y}'(\tfrac{d}{2})\rangle - \langle \nabla w(x), \gamma_{x,y}'(\tfrac{-d}{2})\rangle \leq 2\psi(\tfrac{d}{2})+\textup{tn}_{\underline \kappa}(\tfrac{d}{2}),
 \end{align*}
 where $\psi:[0,D/2]\rightarrow \mathbb R$ satisfies
 \begin{align*}
     \begin{cases}
        0\geq  \psi ''+2\psi \psi'-2\textup
{tn}_{\underline \kappa}\left( \psi '+\psi^2+ \lambda \right)-2\psi\left(\overline{\kappa}-\underline\kappa\right)-4 \psi' \left(\textup{tn}_{\overline \kappa}-\textup{tn}_{\underline \kappa}\right)\\
 0\geq -2(\overline \kappa-\underline \kappa)+2\psi'-4\textup{tn}_{\underline \kappa}\psi-2\textup{tn}_{\underline \kappa}^2+2\textup{tn}_{\underline\kappa}\textup{tn}_{\overline \kappa}\\
 0>\psi'(0)+\frac{\underline\kappa}{2}\\
0\geq \psi'\\
 0=\psi(0).
     \end{cases}
 \end{align*}
  \end{thm}
  \begin{proof}
      This follows from setting $\phi=\psi +\textup{tn}_{\underline \kappa}.$ Indeed, Using 
   $\psi, \psi '\leq 0,$ and  Lemma \ref{TwopointGaussiancomparison}, we have
      \begin{eqnarray*}
        (-4\phi' +2 \textup{tn}_{\underline \kappa}(\tfrac{d(x,y)}{2})\phi)2\mathcal C(x,y) &=& (-4\psi'-4{\underline \kappa} + 2 \textup{tn}_{\underline \kappa}\psi- 2\textup{tn}_{\underline \kappa}^2)2\mathcal C(x,y)  \\
        & \ge &  8 \psi' \textup{tn}_{\overline \kappa} +8\underline\kappa\textup{tn}_{\underline \kappa} - 4 \textup{tn}_{\underline \kappa}^2 \psi  + 4 \textup{tn}_{\underline \kappa}^3,
      \end{eqnarray*}
where we used
\[  \textup{tn}_{\underline \kappa}' =  \underline \kappa + \textup{tn}_{\underline \kappa}^2. \]
     Hence  \begin{align}\label{ineq-1-in-proof-of-moc-thm}
          & -2\phi ''-4\phi \phi'+4\overline\kappa\phi+4\textup{tn}_{\underline \kappa}(\tfrac{d(x,y)}{2})\phi^2+(-4\phi' +2 \textup{tn}_{\underline \kappa}(\tfrac{d(x,y)}{2})\phi)2\mathcal C(x,y)\\ \nonumber
         &  \ge -2(\psi''+2\underline \kappa\textup{tn}_{\underline \kappa} + 2{\textup{tn}_{\underline \kappa}^3})-4(\psi +\textup{tn}_{\underline \kappa})(\psi'+\textup{tn}^2_{\underline \kappa} + \underline{\kappa}) +4\overline \kappa (\psi+\textup{tn}_{\underline \kappa})+4\textup{tn}_{\underline \kappa} (\psi+\textup{tn}_{\underline \kappa})^2 \nonumber \\
         &\ \ \ \ \ \ \ \ \  + 8 \psi' \textup{tn}_{\overline \kappa} +8\underline\kappa\textup{tn}_{\underline \kappa} - 4 \textup{tn}_{\underline \kappa}^2 \psi  + 4 \textup{tn}_{\underline \kappa}^3  \nonumber\\
         & = -2\psi ''-4\psi \psi'+4\textup
{tn}_{\underline \kappa}\left( \psi '+\psi^2 \right)+4\psi( \overline{\kappa} -\underline \kappa) +8 \psi' \left(\textup{tn}_{\overline \kappa}-\textup{tn}_{\underline \kappa}\right)  + 4\overline \kappa\textup{tn}_{\underline \kappa}. \nonumber
         \end{align}
         Now in order to apply Theorem \ref{Propositiongeneralmoc}, we use \eqref{ineq-1-in-proof-of-moc-thm}, Lemma \ref{estiamtesonderivative}, \eqref{curvature-assumption-in-moc-thm} together with \eqref{positivityofcurvature} to obtain the estimate
         \begin{align*}
                & -2\phi ''-4\phi \phi'+4\overline\kappa\phi+4\textup{tn}_{\underline \kappa}(\tfrac{d(x,y)}{2})\phi^2+(-4\phi' +2 \textup{tn}_{\underline \kappa}(\tfrac{d(x,y)}{2})\phi)2\mathcal C(x,y)\\  & +\Bigr(-4\lambda+3\overline\kappa+\textup{tn}_{\overline\kappa}^2(\tfrac {d(x,y)}{2})+ \textup{tn}_{\underline \kappa}(\tfrac{d(x,y)}{2})\mathcal C(x,y)\Bigl ) \mathcal C(x,y)\\
   &-\frac{\varepsilon^2(x,y)}{8\textup{tn}_{\underline \kappa}(\tfrac{d(x,y)}{2})}+\mathcal D(x,y)\\
                \geq &  -2\psi ''-4\psi \psi'+4\textup
{tn}_{\underline \kappa}\left( \psi '+\psi^2 +\lambda\right)+4\psi( \overline{\kappa} -\underline \kappa) +8 \psi' \left(\textup{tn}_{\overline \kappa}-\textup{tn}_{\underline \kappa}\right)  \nonumber\\
&\,-\Bigr(-\overline\kappa+ (\overline \kappa-\underline \kappa) (1+\frac\pi2)+\sqrt 2 \overline \kappa \left(\frac{\overline \kappa-\underline \kappa}{\underline \kappa}\right)^2\left(\frac{4{\overline \kappa}}{ \pi\sqrt{\underline \kappa}}+1\right)^2\Bigl )\textup{tn}_{\underline \kappa}(\tfrac{d(x,y)}{2}) \\
   &\,-2\textup{tn}_{\underline \kappa}(\tfrac{d(x,y)}{2})\left(-\tfrac{2}{\underline \kappa}(\inf\Delta \kappa)_{-}+\left(\tfrac{1}{\underline \kappa^2}+\tfrac{3\pi^2\sqrt8 }{4\overline\kappa\underline \kappa}\right)|\nabla \kappa|^2_\infty+\tfrac{4\pi }{\underline\kappa}|\nabla \kappa|_\infty\right),\nonumber
         \end{align*}
         which is non-negative by assumption.
  Next, we estimate the second inequality of \eqref{inequalities-from-proposition}, using $\phi=\psi +\textup{tn}_{\underline \kappa}.$ We obtain that 
     \begin{align}\label{alpha-terms-in-proof}
         &2 \overline \kappa-2\phi'(\tfrac{d(x,y)}{2})+\textup{tn}_{\underline\kappa}(\tfrac{d(x,y)}{2})\left(2\mathcal C(x,y)+4\phi(\tfrac{d(x,y)}{2})\right)\\
         &\geq 2(\overline \kappa-\underline \kappa)-2\psi'+4\textup{tn}_{\underline \kappa}\psi+2\textup{tn}_{\underline \kappa}^2-2\textup{tn}_{\underline\kappa}\textup{tn}_{\overline \kappa}\nonumber
     \end{align}
     which is positive by assumption. We conclude that the first two inequalities of \eqref{inequalities-from-proposition} hold true. Thus, it remains only to verify the last inequality. To see this, note that 
     \begin{align*}
         \phi'(0)-\tfrac{\underline \kappa}{2}=\psi'(0)+\tfrac{\underline \kappa}{2}<0,
     \end{align*}
     by assumption. 
     From Proposition \ref{Propositiongeneralmoc}, we infer that 
     \begin{align*}
F(x,y)=2\psi(\tfrac{d(x,y)}{2})+2\textup{tn}_{\underline \kappa}(\tfrac{d(x,y)}{2})+\mathcal C(x,y)
     \end{align*}
     is a modulus of concavity for $w=\log u.$ Using Lemma \ref{TwopointGaussiancomparison}, we infer the claim.
  \end{proof}
  \begin{rem}
       Theorem \ref{MOC-thm} holds true under the weaker assumption that for any $x,y\in \Omega$
      \begin{align*}
          4\overline \kappa\textup{tn}_{\underline \kappa}(\tfrac{d(x,y)}{2})+\Bigr(3\overline\kappa+\textup{tn}_{\overline\kappa}^2(\tfrac {d}{2})+ \textup{tn}_{\underline \kappa}(\tfrac{d}{2})\mathcal C(x,y)\Bigl ) \mathcal C(x,y)
   -\frac{\varepsilon^2(x,y)}{8\textup{tn}_{\underline \kappa}(\tfrac{d}{2})}+\mathcal D(x,y)>0.
      \end{align*}
  \end{rem}

  \section{The perturbed Euclidean model}\label{perturbed-euclidean-model-section}

  We now apply Theorem \ref{MOC-thm} to find an explicit modulus of concavity for $w=\log u.$  
  Note that when $\overline \kappa=\underline \kappa,$ we recover Theorem 3.6 in \cite{10.4310/jdg/1559786428}. Namely, following \cite{Dai-Seto-Wei2021}, one obtains that $\psi(s)=(\log\phi_{e})'=\tfrac{\pi}{2D}\tan(\tfrac{2\pi}Ds)$ is a modulus of concavity, where $\phi_{e}$ is the first eigenfunction of the \textit{Euclidean model}
  \begin{align*}
      \phi''=-\lambda\phi \quad \textup{in }[-D/2,D/2]\quad \textup{with }\phi(-D/2)=\phi(D/2)=0.
  \end{align*}

  However, since we are interested in the case where the curvature is not constant, it is necessary to use a different model. 
  Indeed, we consider the \textit{perturbed Euclidean model} 
\begin{align}\label{perturbed-euclidean-model}
 \overline \phi''-4 (\textup{tn}_{\overline \kappa}-\textup{tn}_{\underline \kappa})\overline \phi'=-\overline \lambda \phi \quad \textup{in }[-{D}/2,{D}/2]\quad \textup{ with }\, \overline \phi(-\tfrac {D}2)=\overline \phi(\tfrac {D}2)=0.   
  \end{align}
 
  We now show that $\psi =(\log \overline \phi_{D})' $ satisfies the differential inequalities of Theorem \ref{MOC-thm} where $\overline \phi_{D}$ denotes the first eigenfunction of the problem \eqref{perturbed-euclidean-model}. We divide the proof in several lemma.

We first show that the last two conditions of Theorem \ref{MOC-thm} are satisfied.
\begin{lem}\label{strictlydecreasing}
Suppose that $\overline \phi_{D}$ is the first eigenfunction of the perturbed Euclidean model \eqref{perturbed-euclidean-model} on $[-{D}/2,{D}/2]$
  with Dirichlet boundary conditions. Choosing $\overline \phi_{D}$ to be positive, we find that $\overline\phi_{D}$ and $\psi=(\log \overline \phi_{D})'$ are strictly decreasing on $(0,{D}/2).$
\end{lem}
\begin{proof}
Note that the the first eigenfunction does not switch sign (by the nodal domain theorem) so can be chosen to be positive.  Moreover, since any solution $\phi$ can be replaced by $\widehat\phi(s):=\phi(-s)+\phi(s),$ we find that $\overline \phi_{D}$ is even.
  We now show that $\overline \phi_{D}$ is strictly decreasing on $(0,{D}/2].$ Using an integrating factor, one finds that 
\begin{align*}
    \left( \overline \phi_{D}'\exp(-4\int_0^x (\textup{tn}_{\overline \kappa}-\textup{tn}_{\underline \kappa})\, ds)\right)'=-\overline \lambda _1\exp(-4\int_0^x (\textup{tn}_{\overline \kappa}-\textup{tn}_{\underline \kappa})\, ds)\overline \phi_{D}<0.
\end{align*}
Since $\overline \phi_{D}$ is an even function, we find that $\overline \phi_{D}'(0)=0$ and thus integrating this equation yields the first claim. To show that $\psi $ is strictly decreasing, we calculate \begin{align*}
\psi'=\left(\frac{\phi'}{\phi}\right)' &=\frac{\phi''\phi-(\phi')^2}{\phi^2}\\
    &=4 (\textup{tn}_{\overline \kappa}-\textup{tn}_{\underline \kappa})\psi-\overline{ \lambda_1}-\psi^2 <0.
\end{align*}
\end{proof}
We are now concerned with the first conditions from Theorem \ref{MOC-thm}.
\begin{lem}
     We have the inequality
    \begin{align*}
         0\geq  \psi ''+2\psi \psi'-2\textup
{tn}_{\underline \kappa}\left( \psi '+\psi^2+\overline \lambda \right)-2\psi\left(\overline{\kappa}-\underline\kappa\right)-4 \psi' \left(\textup{tn}_{\overline \kappa}-\textup{tn}_{\underline \kappa}\right)
    \end{align*}
    in $[0,{D}/2).$
\end{lem}
\begin{proof}
     By Lemma \ref{strictlydecreasing},  $\psi$ is negative. As such, combining this with the first inequality of Theorem \ref{MOC-thm}, we find that
\begin{align*}
    4\psi \left(\overline \kappa-\underline \kappa+\textup{tn}_{\overline \kappa}^2-\textup{tn}_{\underline \kappa}^2\right)-8 \textup{tn}_{\underline \kappa}\psi (\textup{tn}_{\overline \kappa}-\textup{tn}_{\underline \kappa})-2\psi\left(\overline{\kappa}-\underline\kappa\right)\leq 0
\end{align*}
if and only if 
\begin{align*}
    2(\overline \kappa-\underline \kappa)+4\textup{tn}_{\overline \kappa}^2+4\textup{tn}_{\underline \kappa}^2-8 \textup{tn}_{\underline \kappa} \textup{tn}_{\overline \kappa}\geq 0.
\end{align*}
The latter of which is true since $\overline \kappa\geq \underline \kappa.$ 
\end{proof}
We now show the second condition.
\begin{lem}  For ${D}<\tfrac{\pi}{2\sqrt{\overline \kappa}}$ and $  \overline \kappa\leq \tfrac{8+4\pi}{3+4\pi}\underline \kappa,$ one has 
\begin{align*}
     2(\overline \kappa-\underline \kappa)-2\psi'+4\textup{tn}_{\underline \kappa}\psi+2\textup{tn}_{\underline \kappa}^2-2\textup{tn}_{\underline\kappa}\textup{tn}_{\overline \kappa}\geq 0.
\end{align*}    
\end{lem}
\begin{proof} Note that 
\begin{align*}
    &  2(\overline \kappa-\underline \kappa)-2\psi'+4\textup{tn}_{\underline \kappa}\psi+2\textup{tn}_{\underline \kappa}^2-2\textup{tn}_{\underline\kappa}\textup{tn}_{\overline \kappa}\\
     =&2(\overline \kappa-\underline \kappa)+2\overline \lambda_1-8 (\textup{tn}_{\overline \kappa}-\textup{tn}_{\underline \kappa})\psi+2\psi^2+4\psi\textup{tn}_{\underline \kappa}+(\textup{tn}_{\underline \kappa}-\textup{tn}_{\overline \kappa})^2-\tfrac{1}{2}{\textup{tn}_{\overline \kappa}^2}.
\end{align*}
Using the identity, $2\psi^2+4\psi\textup{tn}_{\underline \kappa}=2\left(\psi+\textup{tn}_{\underline \kappa}\right)^2-2\textup{tn}_{\underline \kappa}^2$, we find the estimate 
 \begin{align}\label{alpha-terms-ineq}
     &2(\overline \kappa-\underline \kappa)-2\psi'+4\textup{tn}_{\underline \kappa}\psi+2\textup{tn}_{\underline \kappa}^2-2\textup{tn}_{\underline\kappa}\textup{tn}_{\overline \kappa}\\
     &\geq 2(\overline \kappa-\underline \kappa)+2\overline \lambda_1-\tfrac{1}{2}\textup{tn}_{\overline \kappa}^2-2\textup{tn}_{\underline \kappa}^2.\nonumber
 \end{align}
The latter is positive by our diameter and curvature pinching assumption. Now we can bound the first eigenvalue of
\eqref{perturbed-euclidean-model} from below by converting it into a Schr\"odinger equation with Dirichlet boundary condition:
\begin{align*}
  -\varphi''+  \Bigl(4(\textup{tn}_{\overline \kappa}-\textup{tn}_{\underline \kappa})^2-2(\textup{tn}^2_{\overline \kappa}-\textup{tn}^2_{\underline \kappa}+\overline \kappa-\underline \kappa)\Bigr)\varphi=\lambda \varphi.
\end{align*}
From this we get 
\begin{align}\label{lambda_1-estimate}
   \overline  \lambda_1 &=\inf_{u\in H^1_0(- {D}/2,  {D}/2)}\frac{\int_{- {D}/2}^{ {D}/2}(u')^2+\Bigl(4(\textup{tn}_{\overline \kappa}-\textup{tn}_{\underline \kappa})^2-2(\textup{tn}^2_{\overline \kappa}-\textup{tn}^2_{\underline \kappa}+\overline \kappa-\underline \kappa)\Bigr)u^2\, ds }{\int_{-{D}/2}^{ {D}/2}u^2\,ds}\\
   &\geq  \frac{\pi^2}{ {D}^2}-(4+\pi)(\overline \kappa-\underline \kappa) \nonumber
\end{align} 
where we used \eqref{tn_k-square-terms}.
Thus, \eqref{alpha-terms-ineq} can be estimated as follows 
\begin{align}\label{second-ineq}
   &2(\overline \kappa-\underline \kappa)+2\overline \lambda_1-\tfrac{1}{2}\textup{tn}_{\overline \kappa}^2-2\textup{tn}_{\underline \kappa}^2\nonumber\\
   &\geq -(6+2\pi)(\overline \kappa-\underline \kappa)+8\overline \kappa-\tfrac 12 \overline \kappa-2\underline \kappa,
\end{align}
where we used ${D}< \pi /(2\sqrt{\overline \kappa})$. Note that \eqref{second-ineq} is nonnegative whenever
\begin{align*}
  \overline \kappa\leq \frac{8+4\pi}{3+4\pi}\underline \kappa,
\end{align*}
which is our pinch assumption.
\end{proof}
\begin{lem}
    For ${D}\leq \pi/(2\sqrt{\overline \kappa})$ and $\overline \kappa<\tfrac{7+2\pi}{2\pi}\underline \kappa,$
    \begin{align*}
        \psi'(0)+\frac{\underline \kappa}{2}<0.
    \end{align*}
\end{lem}
\begin{proof}
    Observe that $\psi'(0)=-\overline \lambda_1\leq -\tfrac{\pi^2}{{D}^2}+(4+\pi)(\overline \kappa-\underline\kappa) $ by \eqref{lambda_1-estimate}. From the diameter restriction ${D}\leq \pi /(2\sqrt{\overline \kappa
    }),$ one finds that 
    \begin{align*}
        \psi'(0)+\frac{\underline \kappa}{2}\leq -4\overline \kappa +(4+\pi)(\overline \kappa-\underline \kappa)+\frac{\underline \kappa}{2}<0
\end{align*}
by our pinch assumption.
\end{proof}
We thus arrived at the following
\begin{thm}\label{MOC-thm-2}
      Suppose the $(M^2,g)$ is a complete manifold of positive sectional curvature $\kappa$, such that $0<\underline \kappa\leq \kappa\leq \overline \kappa< \tfrac{8+4\pi}{3+4\pi}\underline \kappa$ which satisfies \eqref{curvature-assumption-in-moc-thm}. 
Suppose that $\Omega \subset M^2$ is a geodesically convex domain  such that $\textup{diam}(\Omega)=D<\tfrac{\pi}{2\sqrt{\overline \kappa}}$. Then for any $x,y\in \Omega$, 
 \begin{align*}
     \langle \nabla w(y), \gamma_{x,y}'(\tfrac{d}{2})\rangle - \langle \nabla w(x), \gamma_{x,y}'(\tfrac{-d}{2})\rangle \leq 2(\log \overline \phi_D)'(\tfrac{d}{2})+\textup{tn}_{\underline \kappa}(\tfrac{d}{2}),
 \end{align*}
 where $w=\log u$, the log of the first Dirichlet eigenfunction on $\Omega$ and $\overline \phi_D$ is the first eigenfunction of the perturbed Euclidean model \eqref{perturbed-euclidean-model}.
  \end{thm}

\begin{proof}[Proof of Theorem \ref{MOC-thm-2}]
   We let ${D'}>D$ such that ${D'}<\tfrac{\pi}{2\sqrt{\overline\kappa}}.$ Consider $\psi_{D'} =(\log \overline \phi_{D'})',$ where $ \overline \phi_{D'}$ is the solution of the problem 
   \begin{align*}
 \overline \phi''-4 (\textup{tn}_{\overline \kappa}-\textup{tn}_{\underline \kappa})\overline \phi'=-\overline \lambda \phi \quad \textup{in }[-{D'}/2,{D'}/2]\quad \textup{ with }\, \overline \phi(-\tfrac {D'}2)=\overline \phi(\tfrac {D'}2)=0. 
  \end{align*}
    
   Note also that since ${D'}<\tfrac{\pi}{2\sqrt{\overline\kappa}},$ all the previous lemmas apply to $\psi_{D'}=(\log \overline \phi_{D'})'$ and hence the assumptions of Theorem \ref{MOC-thm} are satisfied (since $\psi_{D'}$ is $C^2$ in $[0,D/2]$). 
   
   As a result, we find that for any $x,y\in \Omega$
   \begin{align*}
       \langle \nabla w(y),\gamma_{x,y}'(\tfrac{d}{2})\rangle - \langle \nabla w(x),\gamma_{x,y}'(\tfrac{d}{2})\rangle \leq 2(\log \overline \phi_{D'})' (\tfrac{d(x,y)}{2})+\textup{tn}_{\underline \kappa}(\tfrac{d(x,y)}{2}).
   \end{align*}
   Then letting ${D'}\rightarrow D$ gives the claim.
\end{proof}

\section{Fundamental Gap Comparison}\label{gapcomparison-section}
In this section, we compare the fundamental gap to the perturbed Euclidean model \eqref{perturbed-euclidean-model} in order to obtain the desired lower bound. However, previously established gap comparison theorems 
do not apply directly in this case, so some modifications are needed.

Since $\textup{tn}_{\overline \kappa}-\textup{tn}_{\underline \kappa} \ge 0$, we can apply a Riccati comparison theorem (see, e.g., \cite{Eschenburg1990}).
\begin{prop}
     Suppose we have two functions  $\psi_1,\psi_2\leq 0$ which satisfy $\psi_1(0) = \psi_2(0)=0$ and
\begin{eqnarray*}
       \psi_1'+\psi_1^2+\lambda-4(\textup{tn}_{\overline \kappa}-\textup{tn}_{\underline \kappa})\psi_1 &=&0, \\
         \psi_2'+\psi_2^2+\lambda &=&0.
\end{eqnarray*}
Then we have 
\begin{align*}
    \psi _1 \leq \psi_2.
\end{align*}
\end{prop}

Applying this to $\psi _1=(\log \overline \phi_1)$ and $\psi_2=(\log \tilde \phi)',$ where $\tilde \phi$ satisfies  \begin{align*}
    \phi''=-\lambda \phi \quad \textup{on }[-L/2,L/2]\quad \textup{ with }\quad \phi(-L/2)=\phi(L/2)=0,
\end{align*}
where $L>0$ is chosen such that $ \tfrac{\pi^2}{L^2}=\overline \lambda_1.$ From this, we obtain the our final log-concavity estimate.
\begin{align}\label{final-superlogconcavity-estimate}
    \langle \nabla w (y),\gamma'(\tfrac{d}{2})\rangle - \langle \nabla w (x),\gamma'(\tfrac{-d}{2})\rangle\leq -\frac{\pi}{L}\tan\left(\frac{\pi}{L}{d(x,y)}\right)+\textup{tn}_{\underline \kappa}\left(\frac{d(x,y)}{2}\right),
\end{align}
where $L=\tfrac{\pi}{\sqrt{\overline \lambda_1}},$ and $\overline \lambda_1$ denotes the first eigenvalue of the perturbed Euclidean model. 

Using this estimate, we can compare the fundamental gap of our domain to that of a one-dimensional interval which is slightly longer, where the length of the interval depends on the pinch. More precisely, we prove that 
\begin{align*}
    \Gamma(\Omega)\geq \frac{3\pi^2}{L^2}.
\end{align*}

In order to establish this, we appeal to a previously known gap comparison theorem.
\begin{thm}[Theorem 4.6 \cite{10.4310/jdg/1559786428}]\label{parabolicgapcomparison}
Let $\Omega \subset M^n$ be convex and suppose that $\textup{Ric}_M\geq (n-1)K.$
Let $v=v(x,t)$ be a solution to the equation 
 \begin{align*}
 \begin{cases}
     \partial _t v= \Delta v+2 \langle \nabla v, X\rangle\\
     \partial _\nu v=0\quad \textup{on }\partial \Omega
     \end{cases}
 \end{align*}
Suppose that $X$ is a time dependent vector field such that
\begin{align*}
    \langle X,\gamma'(\tfrac{d}{2})\rangle - \langle X,\gamma'(\tfrac{-d}{2})\rangle \leq 2w(\tfrac{d}{2})+(n-1)\textup{tn}_K(\tfrac{d}{2}).
\end{align*}
Moreover, assume that $\varphi:[0,D/2)\times \mathbb R_+\rightarrow \mathbb R$ satisfies \begin{itemize}
    \item[$i)$] $\varphi(\cdot, 0)$ is a modulus of continuity of $v(\cdot, 0),$ i.e. 
    \begin{align*}
        |v(y,0)-v(x,0)|\leq 2\varphi(\tfrac d2,0),
    \end{align*}
    \item[$ii)$] $\varphi $ satisfies 
    \begin{align*}
        \frac{\partial \varphi}{\partial t}\geq \varphi ''+2w\varphi'
    \end{align*}
    \item[$iii)$] $\varphi'>0 $ in $[0,D/2)\times \mathbb R_+$\\
    \item[$iv)$] $\varphi(0,t)\geq 0$ for all $t\geq 0.$
\end{itemize}
Then we have that 
\begin{align*}
    |v(y,t)-v(x,t)|\leq 2 \varphi(\tfrac{d}{2},t)\quad \textup{for all }t\geq 0.
\end{align*}
  \end{thm}

With this, we are finally able to give a proof of the main theorem.

\begin{thm}
      Suppose that $(M^2,g)$ is a complete manifold of positive sectional curvature $\kappa$ such that $0<\underline \kappa\leq \kappa\leq \overline \kappa \leq \tfrac{8+4\pi}{3+4\pi} \underline \kappa$ which satisfies \eqref{curvature-assumption-in-moc-thm}. 
Then for any geodesically convex $\Omega \subset M^2$ such that $\textup{diam}(\Omega)=D<\tfrac{\pi}{2\sqrt{\overline \kappa}}$ \begin{align*}
    \Gamma\geq \frac{3\pi^2}{D^2}-(12+3\pi)(\overline \kappa-\underline \kappa).
\end{align*}
  \end{thm}
\begin{proof}
    For $C>0$ large enough, a straightforward computation shows that 
    \begin{eqnarray*}
        \varphi(s,t)=Ce^{-\tfrac{3\pi^2}{L^2}t}\frac{\cos(\frac{2\pi s }{L})}{\cos(\frac{\pi s }{L})} &\quad & s\in  \left ( -\frac{L}{2}, \frac{L}{2} \right ) \\
        v(x,t)=e^{-\Gamma(\Omega)t}\frac{u_2(x)}{u_1(x)} & \quad & x \in \Omega
     \end{eqnarray*}
     satisfy the assumption of Theorem \ref{parabolicgapcomparison}, together with $X=\nabla \log u_1$. As a result, for all $t>0$ we have the estimate
     \begin{align*}
       e^{-\Gamma(\Omega)t}\left(\frac{u_2(y)}{u_1(y)}-\frac{u_2(x)}{u_1(x)}\right)\leq Ce^{-\tfrac{3\pi^2}{L^2}t}\frac{\cos(\frac{d(x,y)\pi  }{L})}{\cos(\frac{d (x,y) \pi}{2L})}
    \end{align*}
    This implies that $\Gamma(\Omega)\geq \frac{3\pi^2}{L^2}.$ Since $\lambda_1=\pi^2/L^2,$ and by  \eqref{lambda_1-estimate}, we conclude 
    \begin{align*}
        \Gamma(\Omega) \geq \frac{3\pi^2}{D^2}-(12+3\pi)(\overline \kappa-\underline \kappa).
    \end{align*}
\end{proof}
  \section{Proof of Proposition \ref{CancellationProposition} and Lemma~\ref{e_2-derivative-of-C}}\label{derivativesection}
  
  In order to prove Proposition~\ref{CancellationProposition}, we begin by computing the derivatives of the variation $\eta$, and then the derivatives of $\mathcal C$. 



\subsection{The derivatives of the variation $\eta$}


Throughout this section, we freely use the variations $\eta_i$ defined by Equation \eqref{Definition_of_eta}. We will also use the fact that the Jacobi fields $J^{1,0}, J^{0,1}$ and $J =J^{1,0} + J^{0,1} $ only have an $e_1$ component, and so we also denote their coefficients with the same notation.
\begin{lem} \label{First half of cancellation} We have the following derivative formulas:
       \begin{align*}
   \langle \nabla _r \nabla _r{\eta'}, e_1\rangle (0,  \tfrac{-d_0}{2})=\int_{\tfrac{-d_0}{2}}^{\tfrac{d_0}{2}}\kappa_1(t)J^{1,0}(t)(J(t))^2\, dt, 
\end{align*}

\begin{align*}
    \langle \nabla _r \nabla _r{\eta'}, e_1\rangle (0, \tfrac{d_0}{2})=-\int_{\tfrac{-d_0}{2}}^{\tfrac{d_0}{2}}\kappa_1(t)J^{0,1}(t)(J(t))^2\, dt.
\end{align*}
  Here $\kappa$ is the sectional curvature and $\kappa_1= e_1(\kappa)$.   
\end{lem}

\begin{proof}  In order to establish these identities, we calculate $\nabla _r \nabla _r{\eta'}(0, \tfrac{d_0}{2})$. To do so, we define an orthonormal frame along $\eta(r,s)$ in the following way:

We parallel transport the vector $e_1$ from $x$ along $\sigma_1(r)$. In other words, we take $e_1 =\tfrac{\partial}{\partial r}\sigma_1(r)$. For $r$ small, $\eta'(r, -\tfrac{d_0}{2}), e_1$ will be a frame along $\sigma_1(r)$. We then define
\begin{equation} \label{TV frame defintion}
    T=\tfrac{\eta_i'(r, -\tfrac{d_0}{2})}{\|\eta_i'(r, -\tfrac{d_0}{2})\|}, \quad V(r)= \frac{\widetilde V(r)} { \|\widetilde V(r)\| },
\end{equation} where $\widetilde V(r) = e_1 - \langle T,e_1\rangle T$.  The pair $\{T, V\}$ is an orthonormal frame along $\sigma_1(r)$ with $V(0) =e_1, T(0)=e_2.$ Finally, we parallel translate $V(r)$ along $\eta(r, \cdot)$, so that $\{T, V\}$ form an orthornormal frame 
which is parallel along each geodesic $\eta(r, \cdot)$.

We now establish that the covariant derivative $ \langle \nabla_r V, e_1 \rangle$ vanish at the $x_0$ and $y_0$. 
\begin{equation}
    \langle \nabla_r V, e_1 \rangle (0, \pm \tfrac{d_0}{2})  =0. \label{V-derivative in r}
    \end{equation}

At $x_0$, this follows from the face that $\sigma_1(r)$ is a geodesic, the identity $\langle T,e_1\rangle=0 $ when $r=0$. For $y_0$, note that 
since \[\overline{V} (r) = \tfrac{\partial}{\partial r}\sigma_2(r) - \langle T(r, \tfrac{d_0}{2}), \tfrac{\partial}{\partial r}\sigma_2(r) \rangle T(r,\tfrac{d_0}{2})\] is orthogonal to $T(r, \tfrac{d_0}{2})$, so the value of $V$ along $\sigma_2(r)$ is also given by $V(r, \tfrac{d_0}{2}) =  \frac{\overline{ V}(r)} { \|\overline{ V}(r)\| }$. Then using the fact that $ \sigma_2(r)$ is a geodesic, this claim follows from the identity $\langle T,e_1\rangle=0 $ at $r=0$.


We now decompose $J$ into terms of the orthonormal frame $\{T, V\}$:
\begin{align*}
    J(r,s)=J_V(r,s)V(r,s)+J_T(r,s)T(r,s)\quad \textup{for }s\in (\tfrac{-d_0}{2},\tfrac{d_0}{2}).
\end{align*}
Doing so, we find that
\begin{align*}
    J'(r,\pm \tfrac{d_0}{2}) =  J_V'(r,\pm \tfrac{d_0}{2})V(r,\pm \tfrac{d_0}{2})+  J'_T(r,\pm \tfrac{d_0}{2}) T(r, \pm \tfrac{d_0}{2}).
\end{align*}

The tangential component of a Jacobi field is linear, so from the boundary condition we have 
 \begin{align*}
    J_T(r,s)=\frac{-1}{d_0}(s-\tfrac{d_0}{2})\Bigl\langle T(r,\tfrac{-d_0}{2}),  \frac{ \partial\sigma_1}{\partial r}(r)\Bigr\rangle+\frac{1}{d_0}(s+\tfrac{d_0}{2})\Bigl\langle T(r,\tfrac{d_0}{2}),  \frac{ \partial\sigma_2}{\partial r}(r)\Bigr\rangle,
\end{align*}
and thus
\begin{align}
J_T(0,\pm \tfrac{d_0}{2})=J_T'(0,\pm \tfrac{d_0}{2})= 0. \label{J_T-2}
\end{align}
Now we compute 
\begin{eqnarray*}
    \nabla_r\nabla_r \eta'(0,\pm \tfrac{d_0}{2})
    &=&\nabla _r\Bigl[ J'_V(r,\pm \tfrac{d_0}{2})V(r,\pm \tfrac{d_0}{2})+  J'_T(r,\pm \tfrac{d_0}{2}) T(r, \pm \tfrac{d_0}{2})\Bigr]_{r=0} \nonumber \\
    & =& \left[\tfrac{\partial}{\partial r} (J'_V) (0,\pm \tfrac{d_0}{2}) e_1+  J'_V(0,\pm \tfrac{d_0}{2})  \nabla _rV (0,\pm \tfrac{d_0}{2}) \right]   + \nabla _r J'_T(r,\pm \tfrac{d_0}{2}) e_2.
\end{eqnarray*}
Using \eqref{V-derivative in r}, we have 
\begin{equation}
 \langle \nabla _r \nabla _r{\eta'}, e_1\rangle (0, \tfrac{d_0}{2}) = \tfrac{\partial}{\partial r} ( J'_V)(0,\pm \tfrac{d_0}{2}). \label{eta-J}
\end{equation}
To compute $\frac{\partial}{\partial r} J'_V$, we differentiate the Jacobi equation 
\[
J_V'' +  J_V \langle R(V, T)T, V\rangle \|\eta'\|^2 = 0
\] with respect to $r$, which gives
\begin{equation}\label{Differentiating Jacobi fields 1}
   \tfrac{\partial}{\partial r} (J_V'') + \tfrac{\partial}{\partial r}(J_V  \kappa(r,s) \|\eta'\|^2 )= 0,  
\end{equation}
where $\kappa(r,s)$ denotes the sectional curvature at the point $\eta(r,s).$
Since the partial derivatives commute and when $r= 0$, $\frac{\partial}{\partial r} \|\eta'\|^2 = \langle J', e_2 \rangle =0$, evaluating Equation \eqref{Differentiating Jacobi fields 1} at $r=0$ gives
\begin{equation}\label{Differentiating Jacobi Fields 2}
    \left(\tfrac{\partial}{\partial r} J_V\right)''(0,s) +  (\tfrac{\partial}{\partial r}J_V) \kappa(0,s) +  J_V (\tfrac{\partial}{\partial r} \kappa) (0,s) =0.
\end{equation}

Since $\nabla_{\tfrac{\partial}{\partial r}} J (0,\pm\tfrac{d_0}{2}) =0 $, with \eqref{J_T-2} and \eqref{V-derivative in r}, we have $\tfrac{\partial}{\partial r} J_V(0,\pm\tfrac{d_0}{2}) =0$. 

As a result, this is a non-homogeneous Jacobi field and we can apply  Lemma~\ref{solutionofsecondorderode} to find the solution
\begin{equation}
\tfrac{\partial}{\partial r} J_V (0,s) = J^{1,0}(s)\int_{-\tfrac{d_0}{2}}^sJ^{0,1}(t)[(J^{1,0})'(\tfrac{d_0}{2})]^{-1}{m(t)}\, dt+J^{0,1}(s)\int_s^{\tfrac{d_0}{2}}J^{1,0}(t)[(J^{1,0})'(\tfrac{d_0}{2})]^{-1}m(t)\, dt,  \label{J-r-derivative-2}
\end{equation}
where $m(t)=-(\tfrac{\partial}{\partial r} \kappa)(0,t)J_V(0,t)=-(\tfrac{\partial}{\partial r} \kappa)(0,t)J(0,t)$ as $J_T(0,t) =0$.
Since $\tfrac{\partial}{\partial r} = J e_1$, we write $\tfrac{\partial}{\partial r} \kappa = J  e_1(\kappa) = J  \kappa_1$. 

Taking the derivative of \eqref{J-r-derivative-2} with respect to $s$ and evaluating this at the end points gives
\begin{eqnarray}
\left(\tfrac{\partial}{\partial r} J_V\right)' (0,-\tfrac{d_0}{2}) &  = &-\left(J^{0,1}\right)'(-\tfrac{d_0}{2})\int_{-\tfrac{d_0}{2}}^{\tfrac{d_0}{2}}J^{1,0}(t)[(J^{1,0})'(\tfrac{d_0}{2})]^{-1}\kappa_1(t) (J(t))^2\, dt  \nonumber \\
 & = &  \int_{-\tfrac{d_0}{2}}^{\tfrac{d_0}{2}} J^{1,0}(t) \kappa_1(t) (J(t))^2\, dt, \label{Nearly done differentiating Jacobi fields} 
 \end{eqnarray}
 \begin{eqnarray}
 \left(\tfrac{\partial}{\partial r} J_V\right)' (0,\tfrac{d_0}{2}) &  = &-\int_{-\tfrac{d_0}{2}}^{\tfrac{d_0}{2}}J^{0,1}(t) \kappa_1(t) (J(t))^2\, dt,   
\end{eqnarray}
where we used Lemma~\ref{HelpforLaplacianTerm} in the second equality. Combining Equation \ref{Nearly done differentiating Jacobi fields} with Equation \eqref{eta-J}, we find the desired result.



\end{proof}

\subsection{Normal derivatives of $\mathcal{C}$}

We now compute the $e_1$ derivative of the function $\mathcal C$ defined in \eqref{def-of-C}.
\begin{lem}
We have 
\begin{align*}
    \nabla _{0\oplus e_1}\mathcal C=-\frac{1}{2}\int_{\tfrac{-d_0}{2}}^{\tfrac{d_0}{2}}\kappa_1J^{0,1}J^2\, d t
\end{align*}
    \begin{align*}
    \nabla _{e_1\oplus 0}\mathcal C=-\frac{1}{2}\int_{\tfrac{-d_0}{2}}^{\tfrac{d_0}{2}}\kappa_1J^{1,0}J^2\, dt.
\end{align*}
\end{lem}
This combines with the previous lemma to prove \eqref{equality1inkeyprop}. 
\begin{proof}
We will focus on $\nabla _{0\oplus e_1}C,$ but the  computation for $\nabla _{e_1\oplus 0}C$ is nearly identical. Let $\sigma(\rho)$ be a curve such that $\sigma(0)=y,$ $\sigma'(0)=e_1(\tfrac d2)$ (to indicate a different variation we use $\rho$ instead of $r$ as before).
Doing so, we have that
\begin{align*}
   \nabla_{0\oplus e_1}\mathcal C(x,y)=\frac{d}{d\rho}_{|_{\rho=0}}\mathcal C(x,\sigma(\rho)). 
\end{align*}
Now, for each $\rho,$ 
we have 
\begin{align*}
\mathcal C(x,\sigma(\rho))=\frac{1}{2}\left(J'(\rho,\tfrac{d(x,\sigma(\rho))}{2})-J'(\rho,-\tfrac{d(x,\sigma(\rho))}{2})\right),
\end{align*}
where $J(\rho,s)$ solves
\begin{align*}
    \begin{cases}
    J''(\rho,s)+ \kappa(\rho,s) J(\rho,s)&=0\\
    J(\rho, -\tfrac{d(x, \sigma(\rho))}{2})=  J(\rho, \tfrac{d(x, \sigma(\rho))}{2})&=1.
    \end{cases}
\end{align*}
Differentiating with respect to $\rho$, commuting the derivatives as before, and evaluating at $\rho=0$, we find that
\begin{align}
    (\tfrac{\partial }{\partial \rho}J)''(0,s)+\kappa(0,s)    \tfrac{\partial J}{\partial \rho}(0,s) +\kappa_\rho(0,s)J=0.  \label{J-rho}
\end{align}
To find the boundary values of $    \tfrac{\partial J}{\partial \rho},$ we note that $J$ at the boundary points is constant (i.e. $1$). So by the first variation formula  $\tfrac{\partial}{\partial \rho}_{|_{\rho=0}}d(x,\sigma(\rho))=0,$ and hence 
\begin{align*}
    0=&\frac{d}{d\rho}_{|_{\rho=0}} J(\rho, \pm \tfrac{d(x, \sigma(\rho))}{2}) \\
    =&\tfrac{\partial J}{\partial \rho}(0, \pm \tfrac{d(x, y)}{2})+J'(0,\pm \tfrac{d(x,y)}{2})\tfrac{d}{d \rho}_{|_{\rho=0}}d(x,\sigma(\rho)) \\
    =&\tfrac{\partial J}{\partial \rho}(0, \pm \tfrac{d(x, y)}{2}).
\end{align*}
Similarly to the computation in the previous lemma, we apply Lemma \ref{solutionofsecondorderode} to equation \eqref{J-rho} to solve for $\tfrac{\partial J}{\partial \rho}(0,s)$, then differentiate this equation with respect to $s$. Evaluating the result at the endpoints, we find that
\begin{align*}
    \tfrac{\partial J'}{\partial \rho}(0, -\tfrac{d_0}{2})-\tfrac{\partial J'}{\partial \rho}(0, \tfrac{d_0}{2}) & 
=  \int_{\tfrac{-d_0}{2}}^{\tfrac{d_0}{2}}J^{1.0} \kappa_\rho(t)J(t)\, dt - \left(- \int_{\tfrac{-d_0}{2}}^{\tfrac{d_0}{2}}J^{0,1} \kappa_\rho(t)J(t)\, dt\right) \\ 
   & =\int_{\tfrac{-d_0}{2}}^{\tfrac{d_0}{2}}\kappa_\rho(t)J^2(t)\, dt  =\int_{\tfrac{-d_0}{2}}^{\tfrac{d_0}{2}}\kappa_1(t)J^{0,1}(t)J^2(t)\, dt
    \end{align*}
  where we have used $J=J^{1,0}+J^{0,1}$ and $\tfrac{\partial }{\partial \rho} =J^{0,1} e_1$ in the last two equalities. 
\end{proof}
Theorem \ref{CancellationProposition} now follows from the previous lemmas.

\subsection{Tangential derivatives of $\mathcal{C}$}
We now compute the $0\oplus e_2$ and $e_2\oplus 0$ derivatives of $\mathcal C$ and thus prove Lemma \ref{e_2-derivative-of-C}.
\begin{proof}[Proof of Lemma \ref{e_2-derivative-of-C}]
We first compute $\nabla _{0\oplus e_2}\mathcal C.$ Note that 
\begin{align*}
    \nabla _{0\oplus e_2}\mathcal C =\frac{d}{dt|_{t=0}}\mathcal C(x,\gamma(\tfrac{d_0}{2}+t)).
\end{align*}
For simplicity, we use the variation $\eta(t,s) = \gamma (s)$, with $s \in [-\tfrac{d_0}{2}, \tfrac{d_0}{2} +t]$. 
For $t$ small, let $J(t,s)$ denote the solution to
\begin{align*}
    \begin{cases}
    J''(t,s)+\kappa(\gamma(s))J(t,s) =0 \quad \textup{for  }s\in(-\tfrac{d_0}{2}, \tfrac{d_0}{2}+t)\\
    J(t, -\tfrac{d_0}{2})=J(t, \tfrac{d_0}{2} +t)=1.
    \end{cases}
\end{align*}
Then \[  2 \nabla _{e_2\oplus 0}\mathcal C= \frac{d}{dt}_{|_{t=0}}\left(J'(t,\tfrac d2 +t)-J'(t, -\tfrac d2)\right)
\]

In this formula, we changed the parametrization, but the solution is independent of the parametrization.
Differentiating at $t=0$ gives $J_t(0,s)$ for $s \in (-\tfrac{d_0}{2}, \tfrac{d_0}{2}).$ This satisfies the same equation, being 
\begin{align}
    J_t''+\kappa J_t=0 \quad \textup{in } (\tfrac{-d_0}{2}, \tfrac{d_0}{2}).  \label{J_t}
\end{align}
We have the boundary conditions:
\begin{align*}
    0=\tfrac{d}{dt}_{|_{t=0}}J(t, \tfrac{-d_0}{2})=J_t(0, \tfrac{-d_0}{2})
\end{align*}
as well as 
\begin{align*}
    0=\tfrac{d}{dt}_{|_{t=0}}J(t, \tfrac{d_0}{2} +t)=J_t(0, \tfrac{d_0}{2})+J'(\tfrac{d_0}{2}).
\end{align*}
In other words, we have that
\begin{align*}
    J_t(0, \tfrac{-d_0}{2})=0 \quad J_t(0, \tfrac{d_0}{2})=-J'(\tfrac{d_0}{2}).
\end{align*}
Applying Lemma~\ref{solutionofsecondorderode} to \eqref{J_t}, we get that 
\begin{align*}
    J_t(0,s)&=-
J'(\tfrac{d_0}{2})J^{0,1}(s),
\end{align*}
which yields
\begin{align}\label{J_t'-difference}
    J_t'(\tfrac{d}{2})-J_t'(-\tfrac{d}{2})&=-J'(\tfrac d2)\left((J^{0,1})'(\tfrac d2)-(J^{0,1})'(-\tfrac d2)\right)\\
    &=-|J'(\tfrac d2)|^2, \nonumber
\end{align}
where in the second equality we used Lemma \ref{HelpforLaplacianTerm}. We thus conclude that 
\begin{align*}
    2 \nabla _{0\oplus e_2}\mathcal C&=
    \frac{d}{dt}_{|_{t=0}}\left(J'(t,\tfrac d2 +t)-J'(t, -\tfrac d2)\right)\\
    &=  J_t'(0,\tfrac{d}{2})-J_t'(0,-\tfrac{d}{2})+J''(\tfrac{d_0}{2})\\
&=-|J'(\tfrac d2)|^2-\kappa(y),
\end{align*}
where we have used \eqref{J_t} in the second equality.\\
We are now computing $\nabla _{e_2\oplus 0}\mathcal C,$ being 
\begin{align*}
    \nabla _{e_2\oplus 0}\mathcal C=
    \frac{d}{dt}|_{t=0}\mathcal C(\gamma(-\tfrac d2+t),y).
\end{align*}With similar arguments as before, we find that $J(t,\cdot)$ is the solution to the problem 
\begin{align}\label{e_2-J(t)}
    \begin{cases}
    J''(t,s)+\kappa(\gamma_{x,y}(s))J(t,s) =0 \quad \textup{for  }s\in(-\tfrac{d_0}{2}+t, \tfrac{d_0}{2})\\
    J(t, -\tfrac{d_0}{2}+t)=J(t, \tfrac{d_0}{2})=1.
    \end{cases}
\end{align}
Then, after differentiating the equations \eqref{e_2-J(t)} for fixed $s\in (-\tfrac d2, \tfrac d2),$ and applying Lemma \ref{inhomogenousjacobimatrix-equ}, we find that 
\begin{align*}
    J_t(0,s)=-J'(-\tfrac d2)J^{1,0}(s).
\end{align*}
From which we get that, using Lemma \ref{HelpforLaplacianTerm}
\begin{align*}
    J_t'(\tfrac d2)-J_t'(-\tfrac d2)=|J'(-\tfrac d2)|^2,
\end{align*}
so that we conclude 
\begin{align*}
    2 \nabla _{e_2\oplus 0}\mathcal C&=
    \frac{d}{dt}_{|_{t=0}}\left(J'(t,\tfrac d2 )-J'(t, -\tfrac d2+t)\right)\\
    &=  J_t'(0,\tfrac{d}{2})-J_t'(0,-\tfrac{d}{2})-J''(-\tfrac{d}{2})\\
&=|J'(-\tfrac d2)|^2+\kappa(x).
\end{align*}
\end{proof}

\section{Second order Derivatives of $\mathcal C$}\label{DerivativesofCSection}

In this section we compute the second order derivatives of $\mathcal C$ and estimate the quantities in terms of curvature and its derivatives, showing that \eqref{estiamtesonderivative} holds.

Recall
\begin{align}\label{DefofEs}
    E_1=e_1\oplus e_1 \in T_x\Omega \oplus T_y\Omega, \quad   E_2=e_2\oplus(-e_2) \in T_x\Omega \oplus T_y\Omega.
\end{align}

\subsection{ $E_2$ derivatives}
We first compute the quantity $\nabla_{E_2,E_2}^2(-\mathcal C)$.

\begin{lem}\label{E_2derivativeofC}
    \begin{align*}
    \nabla_{E_2,E_2}^2(-\mathcal C)&
   =\kappa(x_0)J'(\tfrac{-d_0}{2})-\kappa(y_0)J'(\tfrac{d_0}{2})+\frac{1}{2}\Bigl[-\kappa'(\tfrac{-d_0}{2})+\kappa'(\tfrac{d_0}{2})  \Bigr]\\
    \quad &+[J'(\tfrac{-d_0}{2})]^3-[J'(\tfrac{d_0}{2})]^3-(J^{0,1})'(\tfrac{-d_0}{2})\left(J'(\tfrac{d_0}{2})+J'(\tfrac{-d_0}{2})\right)^2
\end{align*}

\end{lem}


\begin{proof}
Since we are differentiating $\mathcal{C}$ twice in the direction of the geodesic, we consider 
$J(t,\cdot)$ which solve the equation
\begin{align*}
    \begin{cases}
    J''(t,s)+\kappa(\gamma_{x,y}(s)) J(t,s)=0 \quad \textup{for } s\in (-\tfrac d2+t,\tfrac d2-t)\\
    J(t,\tfrac{-d_0}{2}+t)=J(t, \tfrac{d_0}{2}-t)=1.
    \end{cases}
\end{align*}
Doing so, we find the following:
\begin{eqnarray}\label{E_2derivative} 
    \nabla^2_{E_2,E_2}\mathcal C &=&\frac{d^2}{dt^2}_{|_{t=0}}\frac{1}{2}\left( J'(t,\tfrac{-d_0}{2}+t)-J'(t,\tfrac{d_0}{2}-t)\right)  \\
    &=&\frac{1}{2}\left(\underbrace{ J_{tt}'(0,\tfrac{-d_0}{2})-J_{tt}'(0,\tfrac{d_0}{2})}_{I}+\underbrace{J'''(0,\tfrac{-d_0}{2})-J'''(0,\tfrac{d_0}{2})}_{II}\right) \nonumber \\
    & & +\left(\underbrace{J''_t(0,\tfrac{-d_0}{2})+J''_t(0,\tfrac{d_0}{2})}_{III} \right).\nonumber
\end{eqnarray}
To finish the lemma, we compute the terms $I,II,$ and $III$ on the right hand side in terms of curvature and its derivatives.

\textbf{Step 1:} Computing $III.$ \\
Observe that $J_t(0,\cdot)=\tfrac{\partial J}{\partial t}(0,\cdot)$ solves the equation \begin{align*}
    J''+\kappa J=0\quad \textrm{ in }(-\tfrac{d_0}{2},\tfrac{d_0}{2})
\end{align*}
with boundary conditions
\begin{align*}
    J_t(0,\tfrac{-d_0}{2})=-J'(-\tfrac{d_0}{2}), \quad \quad J_t(0, \tfrac{d_0}{2})=J'(\tfrac{d_0}{2}).
\end{align*}
We therefore find
\begin{align}\label{E_22-variationofJ}
    J_t(s)=-J'(\tfrac{-d_0}{2})J^{1,0}(s)+J'(\tfrac{d_0}{2})J^{0,1}(s),
\end{align}
from which it follows that 
\begin{equation} \label{III-E_2}
    J''_t(0,\tfrac{-d_0}{2})+J''_t(0,\tfrac{d_0}{2})= \kappa(x_0)J'(\tfrac{-d_0}{2})-\kappa(y_0)J'(\tfrac{d_0}{2}).
\end{equation}

\textbf{Step 2:} Computing $II.$ \\
 For this term, note that
\[J'''=-(\kappa J)'=-\kappa'J-\kappa J',\] from which we find that 
\begin{equation}\label{II-E_2}
    J'''(0,\tfrac{-d_0}{2})-J'''(0,\tfrac{+d_0}{2})=-\kappa(x_0)J'(\tfrac{-d_0}{2})-\kappa'(\tfrac{-d_0}{2})+\kappa(y_0)J'(\tfrac{d_0}{2})+\kappa'(\tfrac{d_0}{2}).
\end{equation}

\textbf{Step 3:} Computing $I.$ \\
 Observe that $J_{tt}(0,\cdot)$ solves the equation 
\begin{align*}
    J''+\kappa J=0.
\end{align*}
Using the identities
\begin{eqnarray*}
    0&=&\frac{d^2}{dt^2}_{|_{t=0}}J(t, \tfrac{-d_0}{2}+t)=J_{tt}(0,\tfrac{-d_0}{2})+2J_t'(\tfrac{-d_0}{2})-\kappa(x_0)\\
     0&=&\frac{d^2}{dt^2}_{|_{t=0}}J(t, \tfrac{d_0}{2}+t)=J_{tt}(0,\tfrac{d_0}{2})-2J_t'(\tfrac{d_0}{2})-\kappa(y_0),
\end{eqnarray*}
we derive the boundary conditions 
\begin{align*}
  J_{tt}(0,\tfrac{-d_0}{2})=-2J_t'(\tfrac{-d_0}{2})+\kappa(x_0) , \quad J_{tt}(0,\tfrac{d_0}{2})=2J_t'(\tfrac{d_0}{2})+\kappa(y_0).
\end{align*}
From this, we find
\begin{align*}
     J_{tt}(s)=[-2J_t'(\tfrac{-d_0}{2})+\kappa(x_0)]J^{1,0}(s)+[2J_t'(\tfrac{d_0}{2})+\kappa(y_0)]J^{0,1}(s).
\end{align*}
After writing $J=J^{1,1}$, we find that
\begin{align}\label{I-E_2-I}
     J_{tt}'(\tfrac{-d_0}{2})-J_{tt}'(\tfrac{d_0}{2})=[-2J_t'(\tfrac{-d_0}{2})+\kappa(x_0)]J'(\tfrac{-d_0}{2})-[2J_t'(\tfrac{d_0}{2})+\kappa(y_0)]J'(\tfrac{d_0}{2}).
\end{align}
From \eqref{E_22-variationofJ}, we find the following expression for $J_t'(\tfrac{\pm d_0}{2}):$
\begin{eqnarray*}
    J'_t(\tfrac{-d_0}{2})&=&-J'(\tfrac{-d_0}{2})(J^{1,0})'(\tfrac{-d_0}{2})+J'(\tfrac{d_0}{2})(J^{0,1})'(\tfrac{-d_0}{2})\\
     J'_t(\tfrac{d_0}{2})&=&-J'(\tfrac{-d_0}{2})(J^{1,0})'(\tfrac{d_0}{2})+J'(\tfrac{d_0}{2})(J^{0,1})'(\tfrac{d_0}{2}).
\end{eqnarray*}
Combining this and \eqref{I-E_2-I}, we get that 
\begin{eqnarray}
    J_{tt}'(0,\tfrac{-d_0}{2})-J_{tt}'(0,\tfrac{+d_0}{2})& =&[-2\left(-J'(\tfrac{-d_0}{2})(J^{1,0})'(\tfrac{-d_0}{2})+J'(\tfrac{d_0}{2})(J^{0,1})'(\tfrac{-d_0}{2})\right)+\kappa(x_0)]J'(\tfrac{-d_0}{2}) \nonumber \\
   & & -[2\left(-J'(\tfrac{-d_0}{2})(J^{1,0})'(\tfrac{d_0}{2})+J'(\tfrac{d_0}{2})(J^{0,1})'(\tfrac{d_0}{2})\right)+\kappa(y_0)]J'(\tfrac{d_0}{2}).  \label{I-E_2}
\end{eqnarray}
Combining Equations \eqref{III-E_2},\eqref{II-E_2}, and \eqref{I-E_2} into \eqref{E_2derivative}, we find that
\begin{align*}
    \nabla_{E_2,E_2}^2(-\mathcal C)
   &=\kappa(x_0)J'(\tfrac{-d_0}{2})-\kappa(y_0)J'(\tfrac{d_0}{2})+\frac{1}{2}\Bigl[-\kappa'(\tfrac{-d_0}{2})+\kappa'(\tfrac{d_0}{2})  \Bigr]\\
    \quad &\quad \quad +|J'(\tfrac{-d_0}{2})|^2(J^{1,0})'(\tfrac{-d_0}{2})-J'(\tfrac{d_0}{2})J'(\tfrac{-d_0}{2})(J^{0,1})'(\tfrac{-d_0}{2})\\
   & \quad \quad  +J'(\tfrac{d_0}{2})J'(\tfrac{-d_0}{2})(J^{1,0})'(\tfrac{d_0}{2})-|J'(\tfrac{d_0}{2})|^2(J^{0,1})'(\tfrac{d_0}{2}).
\end{align*}
Finally, making use of $J=J^{1,0}+J^{0,1}$ and Lemma \ref{HelpforLaplacianTerm}, we obtain the desired result.
\end{proof}

\subsection{ $E_1$ derivatives}
We now compute the quantity $\nabla_{E_1,E_1}^2(-\mathcal C)$.

\begin{lem}\label{E_1derivativeofC}
    \begin{eqnarray}
   \nabla ^2_{E_1,E_1}(-\mathcal  C)
  =\int \tfrac{J^2(s)}{2}\Bigl[ \kappa_{11}(s)J^2(s)+3\kappa_1(s){p(s)}+\kappa_2(s)q(s)\Bigr]\,ds+\mathcal C(x,y)\nabla_{e_2\oplus(- e_2)} \mathcal C \nonumber
    \end{eqnarray}
where $p(s), q(s)$ are  given by 
\begin{align} \label{p(s) definition}
  p(s)&=-J^{1,0}(s)\int_{\tfrac{-d_0}{2}}^s\tfrac{J^{0,1}(w)}{(J^{1,0})'\left(\tfrac{d_0}{2}\right)}J(w)\kappa_1(w)\,dw-J^{0,1}(s)\int_{s}^{\tfrac{d_0}{2}}\tfrac{J^{1,0}(w)}{(J^{1,0})'\left(\tfrac{d_0}{2} \right)}J(s)\kappa_1(s)\,dw\\
  q(s)&= \tfrac{1}{d_0}(\tfrac{d_0}{2}-s)J'(\tfrac{-d_0}{2}) +\tfrac{1}{d_0}(s+\tfrac{d_0}{2})J'(\tfrac{d_0}{2})-J'(s)J(s)\label{q-definition}
\end{align}
\end{lem}

We make several notes about this computation. We have used $t$ to indicate the variation instead of $r$, to emphasize that the variations are parametrized differently. Unlike the previous expression, it involves integral quantities defined along $\gamma$ so is non-local in nature.

\begin{proof}
Consider the geodesics
\begin{eqnarray*}
    \sigma_x(t) = \exp_x(te_1(-\tfrac d2)) \quad \sigma_y(t) = \exp_y(t e_1 (\tfrac d2))
\end{eqnarray*}
which extend perpendicularly from the endpoints of $\gamma$. For small $t,$ denote $\eta(t, \cdot)$ to be the minimal geodesic between $\sigma_x(t)$ and $\sigma_y(t),$ normalized to have unit speed. 

Doing so, we find that
\begin{align*}
    \frac{d^2}{dt^2}_{|t=0} \mathcal C(\sigma_x(t),\sigma_y(t))=\nabla_{E_1,E_1}^2 \mathcal C(x,y).
\end{align*}
To compute this quantity, we must compute the relevant Jacobi fields and determine their dependence on $t.$ We find that 
\begin{eqnarray}
     \frac{d^2}{dt^2}_{|t=0}J'(t,-\tfrac{d(\sigma_x(t),\sigma_y(t))}{2})&=&J_{tt}'(0,\tfrac{d_0}{2})-J''(0,\tfrac{d_0}{2})\frac{1}{2}\nabla^2_{E_1,E_1}d{(x_0,y_0)}\\
     &=&J_{tt}'(0,\tfrac{-d_0}{2})+\frac{\kappa(x_0)}{2}\nabla^2_{E_1,E_1}d{(x_0,y_0)}. \nonumber
\end{eqnarray}
And similarly,
\begin{align*}
     \frac{d^2}{dt^2}_{|t=0}J'(t,\tfrac{d(\sigma_x(t),\sigma_y(t))}{2})=J_{tt}'(0,\tfrac{d_0}{2})-\frac{\kappa(y_0)}{2}\nabla^2_{E_1,E_1}d{(x_0,y_0)}
\end{align*}
From this, it remains to calculate
$ J_{tt}'(0,\tfrac{-d_0}{2})-J_{tt}'(0,\tfrac{d_0}{2}).$

 $J_{tt}$ satisfies the equation 
\begin{align*}
    \begin{cases} J_{tt}''+\kappa_{tt}J+2\kappa_tJ_t+\kappa J_{tt}=0\\
    J_{tt}(0,\tfrac{-d_0}{2})=\tfrac{1}{2}J'(-\tfrac{d_0}{2})\nabla ^2_{E_1,E_1}d(x,y),\\
    J_{tt}(0,\tfrac{d_0}{2})=-\tfrac{1}{2}J'(\tfrac{d_0}{2})\nabla ^2_{E_1,E_1}d(x,y).
    \end{cases}
\end{align*}

Applying Lemma \ref{solutionofsecondorderode}, we find the following solution for $J_{tt}$:
\begin{align} \label{Jttfirstequation}
    J_{tt}(s)&= \left[\tfrac{1}{2}J'(-\tfrac{d_0}{2})\nabla ^2_{E_1,E_1}d \right ] J^{1,0}(s)+[-\tfrac{1}{2}J'(\tfrac{d_0}{2})\nabla ^2_{E_1,E_1}d]J^{0,1}(s)\\
    &\quad +J^{1,0}(s)\int_{-\tfrac{d_0}{2}}^{s}J^{0,1}(w)\tfrac{m(w)}{(J^{1,0})'(\tfrac{d_0}{2})}\,dw+J^{0,1}(s)\int_s^{\tfrac{d_0}{2}}J^{1,0}(w)\tfrac{m(w)}{(J^{1,0})'(\tfrac{d_0}{2})}\,dw \nonumber
\end{align}
where \[m(w)=-\kappa_{tt}J-2\kappa_tJ_t.\] By the same argument, we find that \begin{align}\label{first-order-jacobivariation}
    J_t(0,s)=p(s).
\end{align} 

Combining Lemma \ref{HelpforLaplacianTerm} with the identities $2\mathcal C(x,y)=\nabla^2_{E_1E_1} d(x,y),$ and $J=J^{1,0}+J^{0,1},$ Equation \eqref{Jttfirstequation} simplifies to the following:
\begin{align} \label{Jttprime}
    J_{tt}'(\tfrac{-d_0}{2})-J_{tt}'(\tfrac{d_0}{2})=\mathcal C(x,y)\Bigl[|J'(\tfrac{d_0}{2})|^2+|J'(\tfrac{-d_0}{2})|^2\Bigr]+\int_{\tfrac{-d_0}{2}}^{\tfrac{d_0}{2}}\underbrace{\kappa_{tt}J^2+2\kappa_tJ_tJ}_{IV}\,dw,
\end{align}
Observe that by Lemma \ref{0opluse_2-terms-of-F}, we get that 
\begin{align*}
    \mathcal C(x,y)\nabla_{e_2\oplus (-e_2)}\mathcal C(x,y)=\frac{1}{2}\mathcal C(x,y)\Bigl(|J'(\tfrac{d_0}{2})|^2+|J'(\tfrac{-d_0}{2})|^2+\kappa(x_0)+\kappa(y_0)\Bigr).
\end{align*}It thus only remains to compute the terms $IV$. Since
\begin{align*}
    \kappa_t(s)=\frac{d}{dt}_{|_{t=0}}\kappa(\eta(t,s))=J(s)\langle \nabla \kappa, e_1\rangle=\kappa_1 J(s),
\end{align*}
and
\begin{align*}
    \kappa_{tt}=\frac{d^2}{dt^2}_{|_{t=0}}\kappa(\eta(t,s))=\langle\nabla_{\tfrac{\partial \eta}{\partial t}(t,s)} \nabla \kappa(\eta(t,s)), \tfrac{\partial \eta}{\partial t}(t,s)\rangle+ \langle \nabla \kappa, \nabla _{\tfrac{\partial \eta}{\partial t}} \tfrac{\partial \eta }{\partial t}\rangle, 
\end{align*}
it remains to compute $\nabla _{\tfrac{\partial \eta}{\partial t}} \tfrac{\partial \eta }{\partial t}$ at $t=0.$ For this, we need the following lemma.

\begin{lem}\label{variation-in-coordinates}
\begin{align*}
    \nabla _{\tfrac{\partial \eta}{\partial t}} \tfrac{\partial \eta }{\partial t}_{|_{t=0}}=p(s)e_1(s)+e_2(s)\left(\tfrac{1}{d_0}(\tfrac{d_0}{2}-s)J'(\tfrac{-d_0}{2}) +\tfrac{1}{d_0}(s+\tfrac{d_0}{2})J'(\tfrac{d_0}{2})-J'(s)J(s)\right)
\end{align*}  
\end{lem}

\begin{proof}[Proof of Lemma \ref{variation-in-coordinates}]
The derivation is very similar to those in Section \ref{derivativesection}. Observe that
\begin{align*}
     \nabla _{\tfrac{\partial \eta}{\partial t}}\tfrac{\partial \eta}{\partial t}_{|_{t=0}} 
    &=\langle \nabla _{\tfrac{\partial \eta}{\partial t}}\tfrac{\partial \eta}{\partial t}_{|_{t=0}}, \tfrac{\partial \eta}{\partial t}\rangle \tfrac{e_1(s)}{J(s)}+\langle \nabla _{\tfrac{\partial \eta}{\partial t}}\tfrac{\partial \eta}{\partial t}_{|_{t=0}},\tfrac{\eta'}{\|\eta'\|}\rangle e_2(s),
\end{align*}
so it suffices to compute $  \langle \nabla _{\tfrac{\partial \eta}{\partial t}} \tfrac{\partial \eta }{\partial t}, \tfrac{\partial \eta}{\partial t}\rangle$ and $  \langle \nabla _{\tfrac{\partial \eta}{\partial t}} \tfrac{\partial \eta }{\partial t}, \tfrac{\eta'}{\|\eta'\|}\rangle.$

Note that 
\begin{align*}
   \langle \nabla _{\tfrac{\partial \eta}{\partial t}} \tfrac{\partial \eta }{\partial t}, \tfrac{\eta'}{\|\eta'\|}\rangle=\frac{d}{dt}_{|_{t=0}}  \langle \tfrac{\partial \eta }{\partial t}, \tfrac{\eta'}{\|\eta'\|}\rangle-J'(s)J(s),
\end{align*}
where we used the fact that $\partial_t\|\eta'\|=0$ at $t=0.$  In order to find $\tfrac{d}{dt}_{|_{t=0}}  \langle \tfrac{\partial \eta }{\partial t}, \tfrac{\eta'}{\|\eta'\|}\rangle,$ we use a similar argument as in the previous section.  As before, $\tfrac{\partial \eta}{\partial t}$ is a Jacobi field which we decompose with an orthonormal frame as $T=\tfrac{\eta'}{\|\eta'\|}$ and an orthogonal part  $V=V(t,s)\perp \eta'.$ We then have that 
\begin{align*}
J(t,s)=J_V(t,s)V+J_T(t,s)T.
\end{align*}
Since \begin{align*}
     \langle \tfrac{\partial \eta }{\partial t}, \tfrac{\eta'}{\|\eta'\|}\rangle=J_T(t,s),
\end{align*}
we have to find 
 $(J_T)_t(0,s).$ With similar arguments as in the previous section, we find that 
\begin{align*}
    J_T''(t,s)=0, \quad J_T(t, \tfrac{-d_0}{2})=\langle \tfrac{\partial \sigma_x}{\partial t}, \tfrac{\eta'}{\|\eta'\| }\rangle(t, \tfrac{-d_0}{2}), \quad J_T(t, \tfrac{d_0}{2})=\langle \tfrac{\partial \sigma_y}{\partial t}, \tfrac{\eta'}{\|\eta'\| }\rangle(t, \tfrac{d_0}{2}),  
\end{align*}
This gives that 
\begin{align*}
    J_T(t,s)=\tfrac{1}{d_0}(\tfrac{d_0}{2}-s)\langle \tfrac{\partial \sigma_x}{\partial t}, \tfrac{\eta'}{\|\eta'\| }\rangle(t, \tfrac{-d_0}{2}) +\tfrac{1}{d_0}(s+\tfrac{d_0}{2})\langle \tfrac{\partial \sigma_y}{\partial t}, \tfrac{\eta'}{\|\eta'\| }\rangle(t, \tfrac{d_0}{2}).
\end{align*}
Thus 
\begin{align*}
    (J_T)_t(0,s)=\tfrac{1}{d_0}(\tfrac{d_0}{2}-s)J'(\tfrac{-d_0}{2}) +\tfrac{1}{d_0}(s+\tfrac{d_0}{2})J'(\tfrac{d_0}{2}).
\end{align*}
As $J_T(0, s)=0$  we have \begin{align*}
    \langle \nabla _{\tfrac{\partial \eta }{\partial t}}\tfrac{\partial \eta}{\partial t}, \tfrac{\partial \eta}{\partial t}\rangle =\frac{1}{2}\tfrac{d}{dt}_{|_{t=0}}\|J(t,s)\|^2=(J_V)_t(0,s)J_V(0,s).
\end{align*}
As in the previous section, we find the variation of Jacobi fields as follows. 
 Note that $J_V(t,s)$ solves the equation $J''+\kappa J=0$ in $(\tfrac{-d(\sigma_x(t),\sigma_y(t))}{2},\tfrac{d(\sigma_x(t),\sigma_y(t))}{2})$ with boundary conditions $J(t, \tfrac{-d(\sigma_x(t),\sigma_y(t))}{2})=J(t, \tfrac{d(\sigma_x(t),\sigma_y(t))}{2})=1.$ Then taking derivative at $t=0$ gives that $(J_V)_t(0,\cdot)$ solves the equation 
\[
J''+\kappa_t J+\kappa J=0\quad  \textup{in}\, (\tfrac{-d_0}{2},\tfrac{d_0}{2}). \]
With boundary condition
    $J_t(0,\tfrac{\pm d_0}{2})=0.$
Then by Lemma \ref{inhomogenousjacobimatrix-equ}, we get $(J_V)_t(0,s)=p(s)$ as desired.
\end{proof}

The proof of Lemma \ref{E_1derivativeofC} follows now from all the previous computations together with Lemma \ref{variation-in-coordinates}. \end{proof}

\begin{prop}\label{E_1+E_2-derivatives-of-C}
\begin{align}
\nonumber    \left(\nabla^2_{E_1,E_1}+\nabla_{E_2,E_2}^2\right)(-\mathcal C)&
=\kappa(x)J_{x,y}'(\tfrac{-d}{2})-\kappa(y)J_{x,y}'(\tfrac{d}{2})+[J_{x,y}'(\tfrac{-d}{2})]^3-[J_{x,y}'(\tfrac{d}{2})]^3\\ \nonumber
    \quad &+\mathcal C(x,y)\nabla_{e_2\oplus (-e_2)}\mathcal C(x,y)+\mathcal D(x,y),\nonumber
\end{align}
where $\mathcal D(x,y)$ is given as follows 
\begin{align}\label{def-of-D}
    \mathcal D(x,y)&=-(J_{x,y}^{0,1})'(\tfrac{-d}{2})\left(J_{x,y}'(\tfrac{d}{2})+J_{x,y}'(\tfrac{-d}{2})\right)^2\\
    &\quad +\int \tfrac{J_{x,y}^2(s)}{2}\Bigl[\Delta \kappa(\gamma_{x,y}(s))J_{x,y}^2(s)+3\kappa_1(\gamma_{x,y}(s)){p(s)}+\kappa_2(\gamma_{x,y}(s))\tilde q(s)\Bigr]\,ds\nonumber,
\end{align}
and $\tilde q(s)=q(s)+4J_{x,y}'(s)J_{x,y}(s).$
\end{prop}

\begin{proof}
    
\begin{align}\label{secondderivativeCequality}
  & \left( \nabla_{E_1,E_1}^2+\nabla ^2_{E_2,E_2}\right)(-\mathcal  C)\\
  = 
    &\int \tfrac{J^2(s)}{2}\Bigl[\Delta \kappa(s)J^2(s)+3\kappa_1(s){p(s)}+\kappa_2(s)\left(\tfrac{1}{d_0}(\tfrac{d_0}{2}-s)J'(\tfrac{-d_0}{2}) +\tfrac{1}{d_0}(s+\tfrac{d_0}{2})J'(\tfrac{d_0}{2})+3J'(s)J(s)\right)\Bigr]\,ds\nonumber
\\
&\quad +\kappa(x_0)J'(\tfrac{-d_0}{2})-\kappa(y_0)J'(\tfrac{d_0}{2})
   +[J'(\tfrac{-d_0}{2})]^3-[J'(\tfrac{d_0}{2})]^3-(J^{0,1})'(\tfrac{-d_0}{2})\left(J'(\tfrac{d_0}{2})+J'(\tfrac{-d_0}{2})\right)^2\nonumber\\
   &\quad +\mathcal C(x_0,y_0)\nabla_{e_2\oplus (-e_2)} \mathcal C, \nonumber 
\end{align}
where we used
\begin{align*}
    \frac{1}{2}\left( \kappa'(\tfrac{d_0}{2})-\kappa'(\tfrac{-d_0}{2})\right)=\frac{1}{2}\int_{\tfrac{-d_0}{2}}^{\tfrac{d_0}{2}}\kappa''J^4+4\kappa'J'J^3(w)\,dw.
\end{align*}
\end{proof}
\begin{rem}
    For a space form $M^2_K$, it is possible to compute these quantities directly.
    Since $C(x,y)=-\textup{tn}_K(\tfrac{d(x,y)}{2}),$ one finds that
    \begin{align*}
        \nabla ^2_{E_2,E_2}(-\mathcal C)=\frac{d^2}{dt^2}_{|_{t=0}} \textup{tn}_K(\tfrac{d}{2}-t)=2\frac{K\textup{tn}_K(\tfrac d2)}{\textup{cs}_K^2(\tfrac d2)}
    \end{align*}
    and that 
    \begin{align*}
        \nabla ^2_{E_1,E_1}(-\mathcal C)=\frac{K}{2\textup{cs}_K^2(\tfrac d2)}\nabla^2_{E_1,E_1}d=-\frac{K\textup{tn}_K(\tfrac d2)}{\textup{cs}_K^2(\tfrac d2)}.
    \end{align*}
   Moreover, we have that $\mathcal D(x,y)=0.$
\end{rem}

\subsection{Proof of Lemma \ref{estiamtesonderivative}}

In this section we estimate the quantities $\mathcal D(x,y)$ and $\varepsilon(x,y)$. Throughout this proof, we will use the fact that $J^{1,0}, J^{0,1},$ and  $J$ are strictly positive in the interval $[-\tfrac{d}{2}, \tfrac{d}{2}].$ As a result, we can apply the comparison results from Section \ref{Jacobifieldcomparisonsection}.

We now estimate $\mathcal D.$ In view of the maximum principle from Section \ref{MOCsection}, we are concerned with a lower bound of $\mathcal D.$ 
\begin{lem}\label{estiamte-of-D} For any $x,y\in \Omega$ with $d(x,y) \leq \pi/(2\sqrt{\overline \kappa})$ 
    \begin{align*}
   \mathcal D(x,y) \geq 2\mathcal C(x,y)\left(\frac{-2(\inf\Delta \kappa)_{-}}{\underline \kappa}+\frac{3\pi^2\sqrt8|\nabla \kappa|^2_\infty }{4\overline\kappa\underline \kappa}+\frac{4\pi|\nabla \kappa|_\infty }{\underline\kappa}+2\sqrt 2\frac{\overline \kappa^2}{\underline\kappa}({\overline \kappa-\underline \kappa})^2\right)
\end{align*}
\end{lem}

\begin{proof}
Concerning the integral terms, note that
\begin{align*}
    |p(s)| &\leq J^{1,0}(s)\int_{\tfrac{-d}{2}}^s\tfrac{J^{0,1}(w)}{-(J^{1,0})'\left(\tfrac{d}{2}\right)}J(w)|\nabla \kappa|(w)\,dw+J^{0,1}(s)\int_{s}^{\tfrac{d}{2}}\tfrac{J^{1,0}(w)}{-(J^{1,0})'\left(\tfrac{d}{2} \right)}J(w)|\nabla \kappa|(w)\,dw\\
    &\leq J(s)\textup{sn}_{\overline \kappa}(d)\Bigl(\int_{-\tfrac{d}{2}}^{\tfrac{d}{2}}J^2|\nabla \kappa|\, dw\Bigr)\\
    &\leq \frac{\textup{sn}_{\overline \kappa}(d)|\nabla \kappa|_\infty}{\textup{cs}^3_{\overline \kappa}(\tfrac{d}{2})}\int_{-\tfrac{d}{2}}^{\tfrac{d}{2}}\textup{cs}^2_{\overline \kappa}(w)\, dw\\&\leq {d^2\sqrt{8}}|\nabla k|_\infty
\end{align*}

and that 
\begin{align*}
   |\tilde q(s)|= &\Bigl|\tfrac{1}{d}(\tfrac{d}{2}-s)J'(\tfrac{-d}{2}) +\tfrac{1}{d}(s+\tfrac{d_0}{2})J'(\tfrac{d}{2})+3J'(s)J(s)\Bigr|\\
   & \leq \left(2+\frac{3}{\textup{cs}_{\overline \kappa}(\tfrac{d}{2})}\right)\textup{tn}_{\overline \kappa}\\
   &\leq \tfrac{5\sqrt{\overline \kappa} d}{\textup{cs}^2_{\overline \kappa}(\tfrac{d}{2})}\leq 10\sqrt{\overline \kappa} d,
\end{align*}
where we used the fact that $\textup{cs}_{\overline \kappa}(\tfrac{d}{2})\geq 1/\sqrt2.$

Thus we obtain that 
\begin{align*}
    &\int \tfrac{J^2(s)}{2}\Bigl[\Delta \kappa(s)J^2(s)+3\kappa_1(s){p(s)}+\kappa_2(s)\tilde q(s)\Bigr]\,ds \\
    \geq &  2d(\inf\Delta \kappa)_{-}-{3\sqrt8|\nabla \kappa|^2_\infty d^3}-10\sqrt{\overline \kappa}|\nabla \kappa|_\infty d^2\\
    \geq & 2\mathcal C(x,y)\Bigl(\frac{-2(\inf\Delta \kappa)_{-}}{\underline \kappa}+\frac{3\sqrt8|\nabla \kappa|^2_\infty d^2}{\underline \kappa}+\frac{10\sqrt{\overline \kappa}|\nabla \kappa|_\infty d}{\underline\kappa}\Bigr)
\end{align*}
where we again used $\textup{cs}_{\overline \kappa}(\tfrac{d}{2})\geq 1/\sqrt2.$
Furthermore, 
\begin{align*}
&-(J^{0,1})'(\tfrac{-d_0}{2})\left(J'(\tfrac{d_0}{2})+J'(\tfrac{-d_0}{2})\right)^2\\
    &\geq  -\tfrac{1}{\textup{sn}_{\overline \kappa}(d)}\textup{tn}_{\underline\kappa}^2(\tfrac{d}{2})\left(\tfrac{\textup{tn}_{\overline \kappa}}{\textup{tn}_{\underline \kappa}}-1\right)^2\\
   &\geq  \sqrt{2}\overline \kappa  \mathcal C(x,y)\left(\tfrac{\textup{tn}_{\overline \kappa}}{\textup{tn}_{\underline \kappa}}-1\right)^2,
\end{align*}
for we estimate 
\begin{align*}
\frac{\textup{tn}_{\overline \kappa}}{\textup{tn}_{\underline \kappa}}-1&= \frac{1}{\textup{tn}_{\underline \kappa}} \int_{\underline \kappa}^{\overline \kappa}\frac{1}{2\sqrt{\kappa}}\tan(\sqrt{\kappa}\tfrac{d}{2})+\frac{d}{4\textup{cs}_\kappa^2(\tfrac{d}{2})}\, d\kappa\\
&\leq \frac{\overline \kappa-\underline \kappa}{\textup{tn}_{\underline \kappa}}\left( \frac{\textup{tn}_{\overline \kappa}(\tfrac{d}{2})}{2 \sqrt{\underline \kappa}}+ \frac{d}{2}\right)\\
&\leq \frac{\overline \kappa-\underline \kappa}{\underline \kappa}\left(\frac{\textup{tn}_{\overline \kappa}(\tfrac{d}{2})}{ d\sqrt{\underline \kappa}}+1\right),
\end{align*}
where we used $\textup{tn}_{\underline \kappa}(\tfrac d2) \geq \underline \kappa\tfrac d2$ in the second inequality. Further, since the function $x \mapsto \tan(x)/x$ is strictly increasing on $(0,\infty)$  (whenever defined), and since $d\leq \tfrac{\pi}{2\sqrt{\overline \kappa}}$
\begin{align*}
    &\frac{\overline \kappa-\underline \kappa}{\underline \kappa}\left(\frac{\textup{tn}_{\overline \kappa}(\tfrac{d}{2})}{ d\sqrt{\underline \kappa}}+1\right)\\
    &\leq \frac{\overline \kappa-\underline \kappa}{\underline \kappa}\left(\frac{4{\overline \kappa}}{ \pi\sqrt{\underline \kappa}}+1\right).
\end{align*}
Putting it all together, with $D\leq \tfrac{\pi}{2\sqrt{\overline \kappa}}$ gives the claim.
\end{proof}
 In order to obtain estimate \eqref{estiamtesonderivative}, we proceed with estimating $\varepsilon(x,y)$ which is defined in \eqref{epsilon-case-1}.
 \begin{lem}\label{estiamte-of-epsilon} For any $x,y\in \Omega$
     \begin{align*}
         - \frac{\varepsilon^2(x,y)}{8\textup{tn}_{\underline \kappa}(\tfrac{d}{2})}\geq  2\mathcal C(x,y)\frac{|\nabla \kappa|^2_\infty}{\underline \kappa^2}.
     \end{align*}
 \end{lem}
 \begin{proof}
Observe that 
\begin{align*}
    \varepsilon(x,y)
    &=\int_{\tfrac{-d}{2}}^{\tfrac{d}{2}}\langle \gamma_{x,y}'(s),\nabla \kappa(s)\rangle\,ds\leq 2d(x,y)|\nabla \kappa|_\infty
\end{align*}
This gives that 
\begin{align*}
     - \frac{\varepsilon^2(x,y)}{8\textup{tn}_{\underline \kappa}(\tfrac{d(x,y)}{2})}\geq -\frac{|\nabla \kappa|_\infty^2d(x,y)}{\underline\kappa }\geq 2\mathcal C(x,y)\frac{|\nabla \kappa|^2_\infty}{\underline \kappa^2}
\end{align*}
as desired.
\end{proof}In order to obtain Lemma \ref{estiamtesonderivative}, we estimate the remaining terms: 
\begin{align*}
&\Bigr(-4\lambda+3\overline\kappa+\textup{tn}_{\overline\kappa}^2(\tfrac {d(x,y)}{2})+ \textup{tn}_{\underline \kappa}(\tfrac{d(x,y)}{2})\mathcal C(x,y)\Bigl )\mathcal C(x,y)\\
\geq & \Bigr(-4\lambda+3\overline\kappa+\textup{tn}_{\overline\kappa}^2(\tfrac {d(x,y)}{2})+ \textup{tn}^2_{\underline \kappa}(\tfrac{d(x,y)}{2})\Bigl )\mathcal C(x,y)\\
\geq &\Bigr(-4\lambda+3\overline\kappa+ (\overline \kappa-\underline \kappa) (1+\frac\pi2)\Bigl )\mathcal C(x,y),
\end{align*}
where we used the fact that due to $\kappa>0,$ $\mathcal C<0$ and moreover,\begin{align}\label{tn_k-square-terms}
    \textup{tn}_{\overline \kappa}^2-\textup{tn}_{\underline \kappa}^2&=\int_{\underline \kappa}^{\overline \kappa}\frac{d}{d\kappa} \textup{tn}^2_\kappa(\tfrac{d}{2})\,d\kappa =\int_{\underline \kappa}^{\overline \kappa}\frac{d}{d\kappa}\kappa\tan^2(\sqrt{\kappa}\tfrac{d}{2})\, d\kappa\\\nonumber
    &=\int_{\underline \kappa}^{\overline \kappa} \left[ \tan^2(\sqrt{\kappa}\tfrac{d}{2})+ \tfrac{d\sqrt{\kappa}}{2} \tan(\sqrt \kappa \tfrac{d}{2})\, \textup{sec}^2(\sqrt \kappa\tfrac{d}{2}) \right] \, d\kappa\\\nonumber
    &\leq (\overline \kappa-\underline \kappa) \left(1+d\sqrt{\overline \kappa}\right)\\\nonumber
    &\leq (\overline \kappa-\underline \kappa) (1+\frac\pi2),
\end{align} 
where we used $D\leq \tfrac{\pi}{2\sqrt{\overline \kappa}}$ in the last two inequalities.
Putting all the above estimates together, Lemma~\ref{estiamtesonderivative} is proved.

\begin{rem}
    It is possible to sharpen the estimate of Lemma~\ref{estiamtesonderivative} with more careful analysis. Doing so will improve the constants, but not essentially change the form of the estimates.
\end{rem}

\bibliographystyle{alpha}
\bibliography{references}

\appendix

\section{Convex domains whose fundamental gap is less than  $\frac{ 3\pi^2+\epsilon}{D^2}$}

\label{Wecannotimproveon3}

In this appendix, we construct domains on a surface which satisfy the estimate 
\[ \Gamma(\Omega) \leq \frac{ 3\pi^2+\epsilon}{D^2}, \]
which implies that the leading term of our estimate cannot be improved.

This proof follows from elliptic regularity, but it requires a bit of care because if one fixes the diameter, sufficiently thin domains actually have very large fundamental gap (See Theorem 1.1 of \cite{surfacepaper1}).

\begin{prop}
     Let $(M^2,g)$ be a Riemannian surface. Then there is a convex domain $\Omega \subset M^2$ which satisfies \begin{equation} \label{Small gaps} \Gamma(\Omega) \leq  \frac{ 3\pi^2+\epsilon}{D^2} \end{equation}
\end{prop}

\begin{proof}

We start by considering four points in $\mathbb{R}^2$ whose convex hull form a rectangle $R$ satisfying \[ \Gamma(R) \leq 
 \frac{3\pi^2+\frac{\epsilon}{4}}{L^2}, \]
where $L$ is the diameter of $R$ in terms of Euclidean distance. 
We use the convention that the long side of $R$ are the \emph{horizontal sides} and the shorter sides of $R$ are the \emph{vertical sides}. Furthermore, we rescale the rectangle so that the horizontal sides have length $1$ in Euclidean distance and center the rectangle at the origin.

We then deform $R$ slightly by replacing all the sides by circular arcs of radius $r \gg 1$, which bow outward. We call this domain $\Omega$, which is now strongly convex in the sense that curvature of its boundary has a positive lower bound. We pick $r$ sufficiently large so that \[\Gamma(\Omega) \leq 
 \frac{3\pi^2+\frac{\epsilon}{2}}{L^2}. \]
Note that the diameter of $\Omega$ might not be exactly the same as the diameter of $R$, but will be very close.

We now consider a geodesic $\gamma$ in $M$ and let $(x,y)$ be Fermi coordinates defined along $\gamma$ (see \cite[Section 3.1]{khan2022negative} for a definition of these coordinates). We then let $\rho$ be a positive parameter and define $\rho \Omega \subset M^2$ to be the image of $\rho \Omega$ in these Fermi coordinates. In other words, we consider the set $\rho \Omega$ as a subset of $\mathbb{R}^2$ and then take the corresponding domain in $M$ under the coordinate chart. Our goal now is to show that for sufficiently small $\rho$, the domain is convex and satisfies the gap estimate \eqref{Small gaps}.

\subsection{Convexity of $\rho \Omega$}
The same section of \cite{khan2022negative} shows that by taking $\rho$ small enough (small enough here depends on bounds on the curvature and its first two derivatives, as well as the height of the original rectangle $R$), the metric is $C^2$ close to a Euclidean metric in the Fermi coordinates. Hence, by taking $\rho\ll \frac{1}{r}$ (so potentially even smaller), all of the circular arcs defining the boundary of $\rho \Omega$ have positive geodesic curvature with respect to the metric $g$. As such, the domain $\Omega$ is convex.

\subsection{Continuity of the spectrum}
To show that the fundamental gap of $\rho \Omega$ is not too large, we perform a rescaling argument. More precisely, we rescale the metric so that the diameter of $\rho \Omega$ is $1$. The rescaling which achieves this is roughly $\frac{g}{\rho}$, but there are some lower order corrections.

Doing so, the curvature and all of its derivatives can be made arbitrarily small (by taking $\rho$ sufficiently small). By elliptic regularity, the spectrum will depend continuously on the coefficients of the metric. As a result, we can choose ${\rho}$ sufficiently small so that the fundamental gap of ${\rho \Omega}$ in the rescaled metric is $\frac{\epsilon}{2}$ from the fundamental gap of the original domain $\Omega$ in the Euclidean metric. As such, in the rescaled metric we have that
\[\Gamma(\rho \Omega) \leq 
 \frac{3\pi^2+\epsilon}{1^2}. \]
However, this estimate is scale-invariant, so returning the metric back to its original scale, we recover the estimate \[\Gamma(\rho \Omega) \leq 
 \frac{3\pi^2+\epsilon}{D^2}. \]
\end{proof}

In practice, what this is doing is taking a rectangle whose fundamental gap is close to $3 \pi^2$ and making it small enough so that the effects of the curvature are negligible. This is a little bit subtle, because if one considers domains of a fixed diameter in a positively curved surface and then shrinks the inradius to zero, the fundamental gap will go to infinity (See Theorem 1.1 of \cite{surfacepaper1}). In particular, this argument requires that we shrink the diameter and inradius to zero simultaneously. 

\begin{rem} \label{3issharpinEuclideanspace}
    In higher dimensions, one can apply a similar construction to obtain convex domains whose gap is $\leq 
 \frac{3\pi^2+\epsilon}{D^2}.$

To do so, we consider a rectangular paralleliped whose sides are all very short except for one. For such a domain in Euclidean space, the fundamental gap can be made arbitrarily close to \ref{gap-3pi}.
Then, we consider a domain $\Omega$ which is an approximation to a long rectangular parallelipiped. However, it is important that this approximation is \emph{strongly convex}, in that there is a positive (albeit small) lower bound on the second-fundamental form of its boundary. From this, the construction can be repeated mutatis mutandis. 
\end{rem}

\end{document}